\documentclass[journal]{IEEEtran}

\usepackage{amsmath,amsfonts}
\usepackage{amssymb}
\usepackage{algorithmic}
\usepackage{algorithm}
\usepackage{array}
\usepackage[caption=false,font=normalsize,labelfont=sf,textfont=sf]{subfig}
\usepackage{textcomp}
\usepackage{stfloats}
\usepackage{url}
\usepackage{verbatim}
\usepackage{graphicx}
\usepackage{cite}
\usepackage{bm}
\usepackage{threeparttable}
\usepackage{color}
\usepackage{amsthm}

\newtheorem{prop}{Proposition}

\begin{document}

\title{An Edge-Based Formulation for the Exact Computation of High-Order Zernike Moments of 2D Shapes and Images}

\author{Patrice~Koehl and Stephan~Tillmann%
\thanks{P. Koehl is with the Department
of Computer Science and Genome Center, University of California,
Davis, CA, 95616.\protect\\
E-mail: koehl@cs.ucdavis.edu}
\thanks {S. Tillmann is with the School of Mathematics and Statistics, The University of Sydney, Sydney, NSW, Australia, 2007.\protect\\
E-mail: stephan.tillmann@sydney.edu.au
}}




\maketitle

\begin{abstract}
Zernike moments are widely used rotation-invariant descriptors for shape and image analysis, but their standard computation relies on a pixel-based quadrature that treats each pixel as a point mass located at its center. 
This approximation introduces spatial aliasing that increases with moment order, degrading image reconstruction and reducing the discriminative power of high-order moments. 
We present an edge-based formulation that eliminates this source of error by applying Green's theorem to transform the two-dimensional area integral defining a Zernike moment into a sum of one-dimensional integrals along image boundaries. 
The resulting framework applies equally to polygonal shapes, binary images, grayscale images, and color images. 
We derive recurrence relations for the required radial primitives and show that the transformed edge integrands are polynomial functions, allowing their exact evaluation using Clenshaw--Curtis quadrature. 
The proposed method computes Zernike moments from polygonal image representations without the spatial discretization errors inherent to conventional pixel-based approaches and remains computationally practical for high-order moments. 
Numerical experiments on image reconstruction, shape analysis, and character classification demonstrate that the proposed formulation matches the accuracy of classical methods at low orders while remaining stable at orders for which pixel-based moments suffer from significant aliasing and numerical degradation.
\end{abstract}


\section{Introduction}
\label{sec:intro}

Moment-based shape descriptors have played a central role in image analysis and pattern
recognition for more than six decades.  The idea of using algebraic moments of image
intensity functions was pioneered by Hu~\cite{Hu:1962}, who showed that a small set
of moment invariants under rotation, translation, and scaling encodes powerful geometric
information about a planar shape.  Shortly thereafter, Teague~\cite{Teague:1980}
proposed replacing ordinary algebraic moments with projections onto a complete orthogonal
basis defined over the unit disk---the Zernike polynomials originally introduced by
Frits Zernike in the context of optical aberration theory~\cite{Zernike:1934}.
The resulting \emph{Zernike moments} inherit the classical desirable properties 
of classical moment invariants, including rotation
invariance, information compactness, and robustness to noise, while adding a crucial practical
advantage: the orthogonality of the basis guarantees that individual moments can be computed
and truncated independently, with no information leakage between coefficients.

The Zernike polynomials $V_n^m(r,\theta) = R_n^{|m|}(r)\,e^{im\theta}$ are defined for
non-negative integer order $n$ and integer repetition $m$ satisfying $|m|\le n$ and
$n-|m|\equiv 0\pmod{2}$.  The radial part $R_n^{|m|}(r)$ is a polynomial of degree $n$
in $r$ that vanishes outside the unit disk~\cite{Zernike:1934,Born:1999}:
\begin{equation}
R_n^{|m|}(r) = \sum_{k=0}^{(n-|m|)/2} (-1)^k F(m,n,k) r^{n-2k},
\label{eq:radial_polynomial}
\end{equation}
where
\begin{equation}
F(n,m,k) = \frac{(n-k)!}{k!\left(\frac{n+|m|}{2}-k\right)!\left(\frac{n-|m|}{2}-k\right)!}.
\end{equation}
The Zernike moment of order $(n,m)$ of a function $f$ supported on (a subset of) the unit
disk is
\begin{equation}
  Z_n^m = \sqrt{ \frac{n+1}{\pi}} \iint_{\mathbb{D}} f(r,\theta)\,(V_n^m)^{\ast}(r,\theta)\,r\,\mathrm{d}r\,\mathrm{d}\theta.
  \label{eq:zernike_moment}
\end{equation}

Since Teague's foundational work, Zernike moments have been applied across a broad spectrum of scientific and engineering disciplines. In optical engineering, they remain the standard language for describing wavefront aberrations and for the design and characterization of optical systems~\cite{Born:1999,Mahajan:1994,Thibos:2002}. In pattern recognition and computer vision, they provide rotation-invariant descriptors for object classification~\cite{Khotanzad:1990,Kan:2002,Zhang:2025} and face recognition~\cite{Singh:2011,Soekarta:2025}. In medical imaging, Zernike descriptors have been used to characterize lesions in mammography and Magnetic Resonance Imaging (MRI)~\cite{Mudigonda:2000,Tahmasbi:2011,Bagherian:2023}, as well as anatomical structures in three-dimensional image data~\cite{Novotni:2003,Novotni:2004}. They have also been applied to image and video watermarking~\cite{Kim:2003,Meng:2025,Chen:2023}, to the analysis of fluorescence three-dimensional spectra in chemistry~\cite{Cui:2025}, and to the characterization of molecular surfaces and cell morphologies in computational biology~\cite{Sael:2008,Kihara:2009,Venkatraman:2009,Banach:2024,Boland:2001,Alizadeh:2016}. These examples represent only a small fraction of the current applications of Zernike moments; comprehensive reviews can be found in~\cite{Kaur:2018,Qi:2021,Niu:2022}.

It is worth noting that image analysis has undergone a profound transformation over the last decade with the emergence of machine-learning methods, and in particular deep neural networks \cite{Krizhevsky:2012, LeCun:2015, Jiao:2019, Montero:2021, Castiglioni:2021, Trigka:2025}. 
For many recognition and classification tasks, learned representations now outperform handcrafted descriptors and have become the dominant paradigm. 
Nevertheless, moment-based descriptors such as Zernike moments continue to occupy an important niche \cite{Qi:2021, Niu:2022}. 
Their appeal stems from several unique properties: they provide compact and interpretable representations of shape, they do not require training data, they possess well-understood mathematical invariance properties, and they can readily be incorporated into machine-learning pipelines as engineered features. 
In applications where explainability, computational efficiency, limited training data, or precise geometric characterization are important, Zernike moments remain highly competitive. 
Consequently, improving the accuracy and numerical stability of their computation remains a relevant and timely problem.

The widespread adoption of Zernike moments has motivated considerable effort toward their efficient and accurate numerical computation. 
The standard approach for computing Zernike moments from a digital image treats the image as a collection of point masses located at pixel centers, each weighted by the corresponding pixel intensity. 
This pixel-based method approximates the continuous integral~\eqref{eq:zernike_moment} by a Riemann sum over pixel centers~\cite{Teague:1980,Khotanzad:1990}. 
While simple to implement and computationally efficient, achieving reliable computations across a wide range of moment orders requires addressing three distinct challenges: the evaluation of the radial polynomials, numerical stability at high order, and the accuracy of the discrete integration.

\paragraph{Radial polynomial computation}
The radial polynomials $R_n^{|m|}$ must be evaluated at every pixel center for each
order $(n,m)$. The explicit summation formula expresses $R_n^{|m|}(r)$ as a sum of
terms with alternating signs and binomial coefficients that grow rapidly with $n$ (see Equation~\ref{eq:radial_polynomial}),
leading to severe cancellation between large numbers of opposite sign. This
catastrophic cancellation makes direct evaluation of the explicit formula numerically
unreliable at even moderate orders. Recurrence relations avoid this problem by
propagating polynomial values from lower orders without ever forming the large
intermediate terms. Kintner~\cite{Kintner:1976} and Prata and
Rusch~\cite{Prata:1989} derived such recurrences, which also reduce the
computational cost from $O(n^2)$ to $O(n)$ per point. A systematic comparison of
recurrence strategies and their numerical properties was carried out by Chong et
al.~\cite{Chong:2003}, and a broader survey is given by Lakshminarayanan
and Fleck~\cite{Lakshminarayanan:2011}. 

\paragraph{Discretization error and the point-mass model}
The pixel-based quadrature treats each pixel as a Dirac $\delta$-function, ignoring the
spatial extent of the pixel cell. This is equivalent to a lowest-order quadrature rule.
For low-order moments the error is tolerable, but as the order $(n,m)$ increases the
radial polynomial $R_n^{|m|}$ begins to oscillate on a scale comparable to or smaller
than a single pixel, and the sampling error becomes of the same magnitude as the moment
itself~\cite{Liao:1998a, Liao:1998b}. Liao and Pawlak~\cite{Liao:1998b}
showed theoretically that this discretization error decreases only algebraically with
pixel size. To reduce it, they proposed replacing each
pixel's point-mass contribution by a polynomial approximation of the integrand over the
pixel cell, yielding a higher-order quadrature without abandoning the pixel-based
framework. A visible consequence of uncorrected discretization error in image reconstruction
is the appearance of \emph{ghost pixels}: spurious nonzero values at locations where the
true image is zero, caused by the aliasing of high-frequency basis functions onto the
coarse pixel grid.

\paragraph{Numerical instability at high order}
Even when the radial polynomials are computed via a stable recurrence, an additional difficulty arises at high orders. The point-mass approximation
accumulates floating-point cancellation errors that grow with $n$, and the two sources of
error---recurrence instability and quadrature noise---compound each other.
Hypergeometric representations~\cite{Janssen:2007} and stabilized three-term
recurrences~\cite{Shakibaei:2013} have been proposed to address the polynomial
evaluation side of the problem, but the quadrature error from the point-mass model
remains. The combined effect makes reliable computation of moments of order
$n \gtrsim 50$ very difficult with standard pixel-based methods on typical image
resolutions.

We note that a further limitation of the pixel-based approach is that it is inherently inapplicable
when the input is a closed planar \emph{contour} rather than a raster image. Contour data
arises naturally from edge detection, segmentation, and vector graphics. Representing
the enclosed region as a binary raster image to apply pixel-based methods reintroduces
all the discretization errors described above. A method that integrates directly over
the contour boundary is therefore both more natural and more accurate in this setting.

In this paper we present a fundamentally different approach to the computation of Zernike
moments that eliminates the discretization error at its source.  Instead of approximating
the integral in Equation~\eqref{eq:zernike_moment} by a sum over pixel centers, we evaluate it
\emph{exactly} over each pixel cell, treating every pixel as a small rectangle over which
the intensity is constant and performing the integration analytically.

The key observation is that the integrand in Equation~\eqref{eq:zernike_moment} can be integrated exactly over a rectangle.
Rather than applying this directly---which would be expensive for high-order moments---we
invoke \textbf{Green's theorem} to convert the 2D area integral over each rectangle into
a line integral along its four edges.  This reduction from a surface to a boundary
integral, standard in potential theory and computational geometry, is here applied to
produce exact moment values at a cost that is manageable via carefully designed recurrences.

Specifically, for a single pixel rectangle $[x_i, x_{i+1}] \times [y_j, y_{j+1}]$, the
contribution to $Z_n^m$ reduces to a sum of 1D integrals of Zernike-related polynomials
along horizontal and vertical segments.  These 1D integrals admit recurrences in both the
order $(n,m)$ and the endpoint coordinates, making the overall algorithm efficient in practice.

The same framework applies directly and naturally to \emph{contour data}: if the region of
interest is bounded by a piecewise-linear closed contour, Green's theorem gives an exact
formula for $Z_n^m$ as a sum of line integrals along the contour edges, without any
reference to a pixel grid.  This makes the method equally applicable to vector-format
shape data.


The principal contributions of this work are as follows.

\begin{enumerate}

\item \textbf{Exact contour-based computation of Zernike moments.}
  We derive a formula for the Zernike moments of the region enclosed by a piecewise-linear
  closed contour, expressed as a sum of line integrals along the contour edges.  By
  Green's theorem, this formula is exact: no discretization or quadrature error is
  introduced.  
  
\item \textbf{Extension to exact integration over raster images.}
  We extend our formalism to computing the Zernike moments of a raster image with piecewise-constant pixels. 
The resulting method computes the \emph{exact} Zernike moments of the
  piecewise-constant image---as opposed to the standard approximation that treats pixels
  as point masses---while remaining computationally practical.

\item \textbf{Comprehensive experimental validation on challenging test cases.}
  We demonstrate three concrete advantages of the proposed method over pixel-based
  computation, namely elimination of artifacts during image reconstruction and stability of high-order moments.
  
\end{enumerate}

The remainder of the paper is organized as follows.
Section~\ref{sec:zernike} reviews the properties of Zernike moments, how they are usually computed for a raster image, and then derives the contour-based integration framework that is the main result of this paper. This framework is valid for
shapes represented by contours, as well as for raster images.
Section~\ref{sec:algo} describes in detail the algorithm for implements
our method. It includes a complexity analysis to compare it with current
methods for computing 2D Zernike moments of an image.
Section~\ref{sec:exp} then presents numerical experiments that validate the approach,
and Section~\ref{sec:conclusion} gives concluding remarks.

\section{2D Zernike moments of 2D shapes and raster images}
\label{sec:zernike}

\subsection{Why Zernike moments?}

Zernike moments represent the projection of an image onto Zernike polynomials, which form a complete and orthogonal set over the unit disk $D = \{(r, \theta) \mid r \leq 1\}$.

The complex Zernike basis functions $V_{n}^m(r, \theta)$ elegantly separate radial and angular components:
\begin{equation}
    V_{n}^m(r, \theta) = \sqrt{ \frac{n+1}{\pi} } R_n^m(r) e^{i m \theta},
\end{equation}
where $n$ is a non-negative integer representing the radial order, $m$ is an integer representing the angular frequency (subject to the conditions $|m| \leq n$ and $n - |m|$ being even), and $R_n^m(r)$ is the real-valued radial polynomial (see introduction),
\begin{equation}
R_n^{|m|}(r)
= \sum_{k=0}^{(n-|m|)/2} (-1)^k F(n,m,k) r^{\,n-2k},
\end{equation}
with
\begin{equation}
F(n,m,k) = \frac{(n-k)!}{k!\left(\frac{n+|m|}{2}-k\right)!\left(\frac{n-|m|}{2}-k\right)!}.
\end{equation}

The continuous Zernike moment $Z_n^m$ of a shape defined by an intensity function $f(r,\theta)$ mapped to the unit disk is given by:
\begin{equation}
    Z_n^m = \iint_D f(r,\theta) (V_n^m)^\ast(r, \theta) \, r \, dr \, d\theta,
\end{equation}
where $\ast$ denotes the complex conjugate.

Zernike moments provide crucial advantages for spatial shape analysis:
\begin{itemize}
    \item \textbf{Orthonormality:} The polynomials are strictly orthogonal over the continuous unit disk, ensuring that each moment captures statistically independent features of the image:
    \begin{equation}
        \iint_D V_n^m(r, \theta) (V_p^q)^\ast(r, \theta) \, r \, dr \, d\theta = \delta_{n,p}\delta_{m,q}.
    \end{equation}
    
    \item \textbf{Completeness:} The basis is complete over the unit disk. An image can be perfectly reconstructed with infinite precision as the maximum expansion order $N \to \infty$.
    \begin{equation}
        f(r,\theta) = \sum_{n=0}^{\infty} \sqrt{ \frac{n+1}{\pi} } \sum_{\substack{m=-n \\ n-|m| \text{ even}}}^{n} Z_n^mV_n^m(r,\theta).
    \end{equation}
    \item \textbf{Rotational Invariance:} Because the angular component is defined by a simple complex exponential, rotating an image by an angle $\alpha$ simply shifts the phase of the moment without altering its magnitude. If $Zr_n^m$ is the moment of the rotated image, then:
    \begin{equation}
        Zr_n^m= Z_n^m e^{-i m \alpha} \implies |Zr_n^m| = |Z_n^m|.
    \end{equation}
    This elegant decoupling of rotation into a pure phase shift makes the magnitude $|Z_n^m|$ a powerful, native rotation-invariant descriptor for pattern recognition.
\end{itemize}

\subsection{Pixel-based Zernike moments of a raster image}
\label{subsec:img}

A 2D image can be seen as a rectangular domain discretized by pixels, namely square cells of constant intensity. The Zernike moments $Z_n^m$ are defined by the area integral over this domain:
\begin{equation}
    Z_n^m = \iint_{D} f( \bm{x} ) (V_{n}^m)^{\ast}(\bm{x}) \, \mathrm{d}\bm{x},
    \label{eqn:imD1}
\end{equation}
where $f$ is the intensity of the image. The traditional method to computing those moments relies on a discretization over the pixels of the image. If there are $N_x$ and $N_y$ pixels along the width and height, respectively, of the image, the discrete Zernike moments of the image are computed as:
\begin{equation}
    Z_n^m = \frac{1}{N_xN_y} \sum_{i=1}^{N_x} \sum_{j=1}^{N_y} I(i,j) R_n^m (r_{x,y}) e^{-im\theta_{x,y}},
    \label{eqn:imD2}
\end{equation}
where $I(i,j)$ is the intensity of the pixel at position $(i,j)$, $(x,y)$ are the Cartesian coordinates of the center of that pixel mapped inside the unit circle, and $r_{x,y}$ and $\theta_{x,y}$ the corresponding polar coordinates. Setting the center of the image at $C=(c_x,c_y) = (N_x/2, N_y/2)$ and its enclosing circle with radius $R = \sqrt{ c_x^2 + c_y^2}$, the mapping is given by
\begin{eqnarray*}
x &=& \frac{(i - c_x -0.5)}{R}, \quad y = \frac{(j - c_y -0.5)}{R}, \\
r_{x,y} &=& \sqrt{ x^2 + y^2}, \quad \theta_{x,y} = \tan^{-1} \left( \frac{y}{x} \right).
\end{eqnarray*}
In this approach, each pixel is represented by a single point with area $1/(N_xN_y)$.

While theoretically robust, the traditional discrete computation of Zernike moments on digital images suffers from severe aliasing and geometric discretization errors. Projecting a continuous circular basis onto a rigid, rectangular pixel grid breaks the strict orthogonality condition. This geometric mismatch directly motivates the need for the exact, piecewise-continuous edge-integration techniques detailed in the following sections.

\subsection{Computing the 2D Zernike moments of a 2D shape}
\label{subsec:contour}

Let us consider a homogeneous shape $S$ occupying a 2D domain $D$. The boundary of the domain, $\partial D$, is defined by a set of closed polygonal contours. This includes an outer boundary and potentially multiple internal boundaries representing holes. The boundary is composed of directed linear edges $e = (A,B)$, ordered counter-clockwise for the external contour and clockwise for internal holes. 

Because Zernike polynomials are complete and orthogonal exclusively over the unit disk $\mathbb{D}$, the domain $D$ must be contained within it ($D \subseteq \mathbb{D}$). In practice, this is achieved by translating the shape's centroid to the origin $O$ and uniformly scaling the shape such that the maximum distance from $O$ to any contour vertex is exactly one.

Assuming the shape is represented by a constant scalar field of unit density, its complex Zernike moments $Z_n^m$ are defined by the area integral:
\begin{equation}
    Z_n^m = \iint_{D} (V_{n}^m)^{\ast}(\bm{x}) \, \mathrm{d}\bm{x},
    \label{eqn:omegaD}
\end{equation}
where $n$ is the radial order, $m$ is the angular frequency subject to $|m| \leq n$ with $n - |m|$ being even.

Using the origin $O$ of the coordinate frame as a reference point, the area integral over the arbitrary polygon $D$ can be exactly decomposed into a sum of signed integrals over oriented triangles (see Figure~\ref{fig:triangle}). Each directed edge $e = (A,B)$ defines an oriented triangle $\sigma_{AB} = (O, A, B)$. The global integral is simply the sum of the integrals over these individual triangles:
\begin{equation}
    Z_{n}^m = \sum_{e \in \partial D} \epsilon_{AB} \iint_{\sigma_{AB}} (V_{n}^m)^{\ast}(\bm{x}) \, \mathrm{d}\bm{x}.
    \label{eqn:omegaT}
\end{equation}
Here, the orientation factor $\epsilon_{AB} = \textrm{sign}(A(\sigma_{AB}))$ is determined by the signed area of the triangle, which is computed via the 2D cross product of the vertex vectors:
\begin{equation}
    A(\sigma_{AB}) = \frac{1}{2} \det(\bm{A}, \bm{B}) = \frac{1}{2} (x_A y_B - y_A x_B).
    \label{eq:area}
\end{equation}

\subsubsection{Reduction to 1D Edge Integrals}

Our task is now reduced to evaluating the integral of the Zernike polynomial over a single triangle $\sigma_{AB}$. We perform this integration by mapping the triangle into polar coordinates $(r, \theta)$ centered at the origin (see Figure~\ref{fig:triangle}). 

\begin{figure*}[t]
\centering
\includegraphics[width=0.85\linewidth]{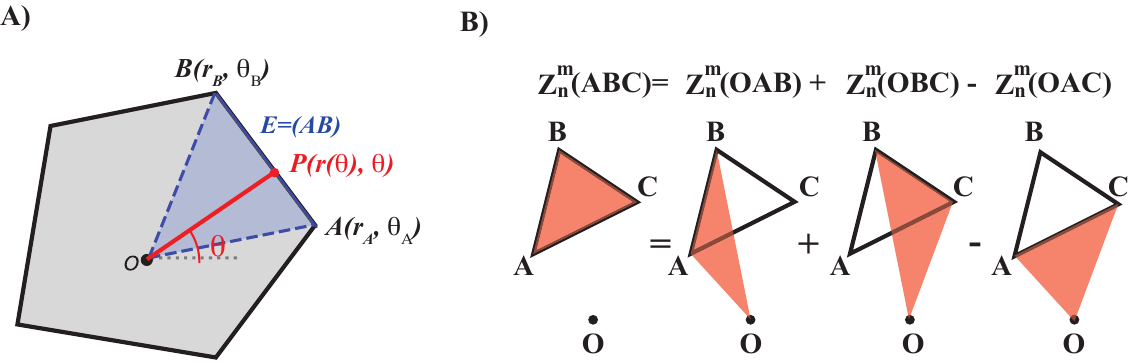}
\caption{\emph{Parameterization of a shape}. (A) A domain D enclosed in a contour defined as a polygonal curve. The Zernike moments of this domain can be computed as the sum of the contribution of all edges along the contour. Let $E=(A,B)$ be such an edge, and $\sigma_{AB}=(O,A,B)$ be the associated triangle formed with the origin $O$. Any ray extending from the origin at angle $\theta$ intersects the edge $E$ at a specific point $\bm{P}$ with polar coordinates $(r(\theta),\theta)$. The Zernike moment can then be computed through separation of variables, with first a radial integration followed by an angular integration along the edge. (B) The decomposition remains true for shapes that do not contain the origin, as long as the edge contributions are signed.}
\label{fig:triangle}
\end{figure*}

Any point $\bm{M}$ with polar coordinates $(r, \theta)$ inside the triangle lies on a ray that intersects the bounding edge $e$ at a boundary point $\bm{P}$. Let the radial distance to this boundary point be defined as the function $r(\theta)$. The 2D area integral over the triangle can then be separated into a sequence of two 1D integrals: first a radial integration from the origin to the edge $r(\theta)$, followed by an angular integration sweeping from $\theta_A$ to $\theta_B$:
\begin{eqnarray*}
    Z_{n}^m(e) &=& \epsilon_{AB} \iint_{\sigma_{AB}} (V_{n}^m)^{\ast}(\bm{x}) \, \mathrm{d}\bm{x} \\
    &=& \epsilon_{AB} \int_{\theta_A}^{\theta_B} \left( \int_{0}^{r(\theta)} R_n^m(r) r \, \mathrm{d}r \right) e^{-im\theta} \, \mathrm{d}\theta.
\end{eqnarray*}

This nested expression defines an efficient two-step analytical process for evaluating the Zernike moment contribution of any edge:
\begin{itemize}
    \item[i)] \textbf{Radial Integration (Anti-derivative):} For the boundary point $\bm{P}$ located at distance $r(\theta)$ from the origin, define the exact radial primitive $Q_n^m$:
    \begin{equation*}
        Q_{n}^m(r(\theta)) = \int_{0}^{r(\theta)} r R_{n}^m(r) \, \mathrm{d}r.
    \end{equation*}
    
    \item[ii)] \textbf{Angular Integration (Line Integral):} Integrate this evaluated primitive along the angular span of the edge $e$:
    \begin{equation*}
        Z_{n}^m(e) = \epsilon_{AB} \int_{\theta_A}^{\theta_B} Q_{n}^m(r(\theta)) e^{-im\theta} \, \mathrm{d}\theta.
    \end{equation*}
\end{itemize}

In the following sections, we derive exact recurrence relationships to efficiently evaluate the radial primitives $Q_n^m(r)$, and describe a robust numerical method to compute the final line integral for $Z_{n}^m(e)$ using Clenshaw-Curtis quadrature.

\subsubsection{Computing the integrals $Q_{n}^m(r)$}

We consider the integrals $S_{n}^m(r)$ and $Q_{n}^m(r)$ of the radial functions $R_{n}^m(r)$:

\begin{eqnarray*}
S_n^m(r) &=& \int_{0}^{r} R_{n}^m(t) \mathrm dt, \\
Q_n^m(r) &=& \int_{0}^{r} t R_{n}^m(t) \mathrm dt.
\end{eqnarray*}

In Appendix \ref{sec:Appendix} we validate the following proposition that allows us to compute $R_n^m(r)$, $S_n^m(r)$, and $Q_n^m(r)$ recursively:
\begin{prop}
For non-negative integers $n\ge 2$ and $m$ with $2 \le m \le n$ and $n-m$ even, the following relationships hold:
\begin{eqnarray}
R_n^m(r) &=& r \left( R_{n-1}^{| m-1 |}(r) + R_{n-1}^{m+1}(r) \right) - R_{n-2}^m(r) , \nonumber \\
\label{eqn:rec1}  \\
S_n^{m-2}(r) &=& \frac{1}{n+1} \left( R_{n+1}^{m-1}(r) - R_{n-1}^{m-1}(r) \right) - S_{n}^m(r) , \nonumber \\
\label{eqn:rec2} \\
Q_{n}^{m-2}(r) &=& S_{n+1}^{m-1}(r) + S_{n-1}^{m-1}(r) - Q_{n}^m(r),\label{eqn:rec3}
\end{eqnarray}
with the additional initializations for $m=n$:
\begin{eqnarray*}
R_{n}^n(r) = r^n, \quad S_n^n(r) = \frac{r^{n+1}}{n+1}, \quad \text{and} \quad Q_n^n(r) = \frac{r^{n+2}}{n+2}.
\end{eqnarray*}
\end{prop}

It is noteworthy that $R_n^m(r)$, $S_n^m(r)$, and $Q_n^m(r)$ are polynomials of degree $n$, $n+1$, and $n+2$, respectively.

\subsubsection{Exact evaluation of the angular integral via Quadrature}

Recall that the Zernike moment of an edge $e$ is expressed as an integration along the angular span of the edge $e$:
    \begin{equation}
        Z_{n}^m(e) = \epsilon_{AB} \int_{\theta_A}^{\theta_B} Q_{n}^m(r(\theta)) e^{-im\theta} \, \mathrm{d}\theta.
     \label{eq:znme}
    \end{equation}

To compute the angular integral efficiently and accurately, we transition from the polar angle $\theta$ to a 1D Cartesian parameterization. 

We parameterize the line segment connecting vertices $\bm{A} = (x_A, y_A)$ and $\bm{B} = (x_B, y_B)$ using a standard scalar parameter $t \in [-1, 1]$. Any point $\bm{P}(t)$ on the edge $e$ is defined linearly as:
\begin{equation*}
    \bm{P}(t) = \frac{1 - t}{2} \bm{A} + \frac{1 + t}{2} \bm{B},
\end{equation*}
i.e., 
\begin{eqnarray}
 x(t) &=& \frac{ (1-t) x_A  + (1+t) x_B}{2},  \nonumber \\
 y(t) &=&\frac{ (1-t) y_A  + (1+t) y_B}{2} .
\label{eq:xy}
\end{eqnarray}
Under this Cartesian parameterization, the radial distance is $r(t) = \sqrt{x(t)^2 + y(t)^2}$. The angular differential $\mathrm{d}\theta$ can be derived from the Cartesian coordinates as:
\begin{equation*}
    \mathrm{d}\theta = \frac{x(t)y'(t) - y(t)x'(t)}{r(t)^2} \, \mathrm{d}t.
 \end{equation*}
 Using the fact that $x(t)$ and $y(t)$ are linear functions of $t$ (see Equation~\ref{eq:xy}), we get:
 \begin{eqnarray}
  \mathrm{d}\theta &=&  \frac{x_A y_B - y_A x_B}{2 r(t)^2} \, \mathrm{d}t 
 = \frac{C dt}{r(t)^2},
\end{eqnarray}
where  we have set $C= (x_Ay_B-y_Ax_B)/2$. Substituting these into our 1D edge integral in Equation~\ref{eq:znme} maps the domain from $[\theta_A, \theta_B]$ to the standard interval $[-1, 1]$:
\begin{eqnarray}
    Z_{n}^m(e) &=& \epsilon_{AB} \int_{-1}^{1} Q_{n}^m(r(t)) e^{-im\theta(t)}  \frac{\mathrm{d}\theta}{\mathrm{d}t} \, \mathrm{d}t \nonumber \\
    &=& \epsilon_{AB} \int_{-1}^{1} A_{n}^m(t) \, \mathrm{d}t,
\end{eqnarray}
where
\begin{equation}
A_n^m(t) = C\frac{ Q_{n}^m(r(t)) e^{-im\theta(t)} }{r(t)^2}.
\label{eq:Anm}
\end{equation}
We prove now the following property:
\begin{prop}
The integrand $A_n^m(t)$ is a polynomial function of $t$ of order at most $n$.
\end{prop}
\begin{proof}
We will focus on $m \ge 0$; the proof is exactly the same for $m <0$. Note first that the radial polynomial $R_n^m(r)$ is exactly equivalent to a shifted Jacobi polynomial.
The relation is defined as:
\begin{equation}
R_{n}^{m}(r)=(-1)^{\frac{n-m}{2}} r ^{m}P_{\frac{n-m}{2}}^{(0,m)}(1-2r ^{2}),
\label{eq:jacobi}
\end{equation}
where \(P_k^{(\alpha, \beta)}(x)\) represents the standard Jacobi polynomial of degree $k$ with parameters $\alpha$ and $\beta$.
To satisfy the polynomial definition, the indices $n$ and $m$ must obey the following rules: $n \ge m$ and $n - m$ must be an even integer, which are exactly the rules under which Zernike polynomials are defined.
We rewrite Equation~\ref{eq:jacobi} as:
\begin{equation}
R_{n}^{m}(r)=r ^{m}P_n^m(r^2),
\end{equation}
where where $P_n^m$ is a polynomial of degree $(n-m)/2$ in its argument $r^2$. Replacing in $Q_n^m(r)$,
\begin{equation*}
Q_{n}^{m}(r)=\int_0^r t^{m+1}P_n^m(t^2) dt,
\end{equation*}
and using the polynomial expansion of $P_n^m$, it is then easy to show that
\begin{equation}
Q_{n}^{m}(r)= r^{m+2} V_n^m(r^2),
\label{eq:qv}
\end{equation}
where $V_n^m$ is a polynomial of degree $(n-m)/2$ in its argument $r^2$.
Replacing Equation~\ref{eq:qv} into Equation~\ref{eq:Anm}, we get:
\begin{eqnarray}
A_{n}^{m} &=& C\frac{ Q_{n}^m(r(t)) e^{-im\theta(t)} }{r(t)^2} \nonumber \\
&=& C \frac {r(t)^{m+2} V_n^m(r(t)^2) e^{-im\theta(t)} }{r(t)^2} \nonumber \\
&=& C V_n^m(r(t)^2) r(t)^m e^{-im\theta(t)} .
\end{eqnarray}
Recall that $r(t)=\sqrt{x(t)^2 + y(t)^2}$. In addition, 
\begin{equation*}
r(t)^m e^{-im\theta(t)}  = (x(t)-i y(t))^m.
\end{equation*}
Therefore,
\begin{equation*}
A_{n}^{m}(t) = V_n^m(x(t)^2 + y(t)^2) (x(t)-i y(t))^m.
\end{equation*}
As both $x(t)$ and $y(t)$ are linear functions of $t$, and as $V_n^m$ is a polynomial of order at most $\frac{n-m}{2}$ with argument $r^2$, this concludes the proof that $A_{n}^{m}(t)$ is a polynomial in $t$. In addition, since $x(t)^2+y(t)^2$ is a quadratic polynomial in $t$,
the factor $V_n^m(x(t)^2+y(t)^2)$ has degree at most $n-m$ in $t$,
while $(x(t)-iy(t))^m$ has degree $m$.
Their product therefore has degree at most $(n-m)+m=n$.
\end{proof}

Because the transformed integrand $A_{n}^m(t)$ is a polynomial of maximum degree $n$, the integral can be evaluated \textit{exactly} using numerical quadrature. 
While Gauss--Legendre quadrature requires the theoretical minimum number of nodes for a single polynomial, we utilize the \textbf{Clenshaw-Curtis} quadrature rule for superior computational efficiency in practice. 
Although Gauss--Legendre quadrature requires fewer nodes for a fixed polynomial degree, the nested structure of Clenshaw--Curtis nodes allows previously computed function evaluations to be reused when adaptive refinement is needed. 
In practice, this advantage outweighs the slightly larger
number of quadrature points. Details will be provided in Section~\ref{sec:algo}.

\subsection{Generalization to Arbitrary 2D Meshes and to raster images}

The formulation derived in the preceding section, which transforms the 2D area integral into a 1D summation over boundary edges, is highly versatile. Because the reduction relies fundamentally on Green's theorem, the methodology is not restricted to simple, continuous analytical contours but generalizes naturally to any arbitrary 2D mesh representation.

For instance, if a 2D shape is represented by a triangular mesh or an arbitrary polygonal tessellation, the total Zernike moment of the shape is exactly evaluated as the sum of the moments computed over all oriented edges of the mesh. For a shape with uniform intensity, the line integrals along internal edges shared by adjacent geometric elements perfectly cancel each other out due to their opposing orientations. This ensures that only the true external boundaries (and the boundaries of any internal holes) contribute to the final moment, preserving both mathematical exactness and computational efficiency.

This new formalism also provides an alternate approach to computing the Zernike moments of an image. Indeed, the edge-based integration scheme developed above adapts seamlessly to digital, pixel-based images. 
A 2D raster image can be treated as a specialized uniform square mesh, where each pixel acts as a piecewise constant square domain of given intensity $I$. 
When computing the Zernike moments over this grid, the integration is performed along the horizontal and vertical edges separating the pixels.
Because the pixels in a 2D image possess a consistent topological orientation, adjacent pixels traverse their shared boundaries in opposite directions. 
Consider, for example, a horizontal edge $e$ forming the bottom boundary of a pixel $p_1$ and the top boundary of a pixel $p_2$. 
If the contour integral for $p_1$ evaluates $e$ in a positive direction, the integral for $p_2$ evaluates the exact same edge in the negative direction. 
Since the total integral over the image is the sum of the integrals over each individual pixel, these adjacent boundary evaluations are combined.

In this discrete context, the net contribution of any given edge $e$ shared by two adjacent pixels --such as $p_1$ and $p_2$-- is naturally scaled by the directional intensity gradient across that edge:
\begin{equation}
\Delta I_e = I(p_1) - I(p_2).
\end{equation}
Consequently, an edge contributes to the total Zernike moment if and only if there is a difference in intensity between its two neighboring pixels. Edges separating pixels of identical color or grayscale intensity yield $\Delta I_e = 0$ and naturally vanish from the summation.

\section{Algorithm for computing the 2D Zernike moments for an image or a 2D shape}
\label{sec:algo}

The previous section provides the elements for computing the 2D Zernike moments
of a raster image through continuous integration over the edges between the pixels of this image.
We summarize those elements in Algorithm~\ref{algo:alg1}.

\begin{algorithm}[htb] 
\caption{\emph{Edge-based Zernike moments of a raster image}}
\label{algo:alg1}
\begin{algorithmic}
\STATE \textbf{Input:}  The image with $N_x \times N_y$ pixels. $N$: The maximum order for the  2D Zernike moments. 
\STATE \textbf{Initialize:}  Set coordinate system such that the image fits in the unit circle. Define set of edges $S=\{e\}$  between pixels and set their weights $I(e)$. Initialize $Z_n^m=0$ for all $0 \le n \le N$, $0 \le m \le n$, with $n-m$ even.
\FOR {$e \in S$}
	\STATE	(1)	Define end points $A$ and $B$ for $e$, its orientation $\epsilon_{AB}$, and its weight $\Delta I_e$.
	\STATE  (2)     Compute contribution of edge $e$:
		\begin{equation*}
			Z_{n}^m(e) = \Delta I_e \epsilon_{AB} \int_{-1}^{1} A_{n}^m(t) \, \mathrm{d}t
    		\end{equation*}
		using the Clenshaw-Curtis quadrature algorithm.
	\STATE (3)       Update $Z_n^m = Z_n^m + Z_{n}^m(e)$ for all possible $n$ and $m$.
\ENDFOR
\STATE	\textbf{Output:} The Zernike moments $Z_{n}^m$ associated with the image.
\end{algorithmic}
\end{algorithm}

We note first that this algorithm can be easily expanded to analyzing 2D shapes defined by contours, usually an external boundary and possibly many inner holes.
Each contour is represented by a polygon. 
External boundaries are oriented counterclockwise, whereas internal holes are oriented clockwise. 
The edges of all the polygons define the set $S$ in the initialization step of the algorithm.

The loop over the edges of either an image or a contour is embarrassingly parallel and can therefore be distributed efficiently across multiple CPU cores or GPU threads.

There are two essential elements in the algorithm that have not yet been discussed: (i) the quadrature used to perform the continuous integrations over the edges,
and (ii) the definition of a pixel intensity for color images. These two elements are discussed below. We then end this section with a discussion on complexity. 
\subsection{Clenshaw-Curtis quadrature}

The primary advantage of Clenshaw-Curtis quadrature is that its nodes are perfectly nested when the sequence size is defined as $N_c = 2^k + 1$. This nesting allows for highly efficient adaptive integration: if the precision must be increased, the number of nodes can be doubled without discarding any of the previously evaluated points. Furthermore, when computing an entire set of moments up to $N_{max}$, the same set of coordinate evaluations $\bm{P}(t_k)$ can be universally reused for all required polynomials.

A Clenshaw-Curtis rule utilizing $N_c$ points integrates exactly any polynomial of degree up to $N_c - 1$. Therefore, choosing the number of quadrature points such that $N_c \geq n + 1$ guarantees an exact mathematical result.

The continuous integral is thus replaced by a finite sum over the optimally chosen nodes:
\begin{equation}
    Z_{n}^m(E) = \sum_{k=0}^{N_c-1} w_k F_{n}^m(t_k).
    \label{eq:z_e}
\end{equation}
The Clenshaw-Curtis nodes $t_k \in [-1, 1]$ are defined as the extrema of Chebyshev polynomials (the Chebyshev-Gauss-Lobatto nodes), given analytically by:
\begin{equation*}
    t_k = \cos\left( \frac{k \pi}{N_c - 1} \right),
\end{equation*}
and $w_k$ are the corresponding quadrature weights derived from exactly integrating the interpolating Chebyshev polynomial.

In practice, for each pre-computed Clenshaw-Curtis node $t_k$, we simply evaluate the Cartesian coordinates $\bm{P}(t_k)$ along the edge, convert them into the required polar components $(r(t_k), \theta(t_k))$, and accumulate the weighted sum simultaneously across all necessary moments.

\subsection{Complex Zernike Moments (CZM) and Quaternion Zernike Moments (QZM).}

There are 3 main types of images: binary, grayscale, and color. 
For the first two types, a pixel is defined with one channel, with values of 0 (black) or 1 (white) for a binary image, and with values of 0 to 255 for a gray scale of 255 values between black and white. For both cases, complex Zernike moments are computed, as described above.

Color images define each pixel over three channels, Red, Green, and Blue. While it is possible to consider the 3 channels separately, such an integration fails to capture the holistic correlation between RGB components. Instead, we represent a color pixel as a pure quaternion \cite{Chen:2012,Chen:2012b,  Chen:2015, Huang:2023}.:
\begin{equation}
    f(x,y) = R(x,y)i + G(x,y)j + B(x,y)k.
\end{equation}
The right-sided Quaternion Zernike Moment replaces the complex exponential with a quaternion exponential using the space diagonal axis $\mu = \frac{i+j+k}{\sqrt{3}}$:
\begin{equation}
    QZ_{n}^{m} = \frac{n+1}{\pi} \iint_D f(x,y) R_{n,m}(r) e^{-\mu m \theta} \, dx \, dy.
\end{equation}

\subsection{Computational Complexity and Accuracy Trade-offs.}

When evaluating the computational efficiency of the proposed edge-based integration against traditional pixel-based summation, a careful analysis of algorithmic complexity is required. 

Let $N$ be the maximum radial order of the Zernike moments, and let an image consist of $P$ total pixels and $E$ active boundary edges (where the intensity gradient $\Delta I \neq 0$). 

The traditional discrete pixel-based approach evaluates the spatial recurrences at the center point of each pixel. Because there are roughly $N^2/2$ moments to compute up to order $N$, the computational cost per pixel is $\mathcal{O}(N^2)$, resulting in a global complexity of $\mathcal{O}(P \cdot N^2)$. However, this is a zeroth-order geometrical approximation; point-sampling a polynomial of degree $N$ over a finite square area introduces severe discretization errors and high-frequency aliasing. 

To make the area-based pixel integration mathematically exact—matching the precision of our continuous contour formulation—one would have to apply 2D Gauss--Legendre quadrature over the area of each pixel. Exact integration of a 2D polynomial of degree $N$ requires $\mathcal{O}(N^2)$ quadrature nodes per pixel. Evaluating the $\mathcal{O}(N^2)$ moments at $\mathcal{O}(N^2)$ nodes results in an exact area integration complexity of $\mathcal{O}(P \cdot N^4)$.

By applying Green's Theorem, our formulation reduces the 2D area integral to a 1D line integral. The required number of 1D Gauss quadrature points along an edge is $N_g \propto N$. Evaluating the $\mathcal{O}(N^2)$ moments at $N_g$ points yields a complexity of $\mathcal{O}(E \cdot N^3)$. 
Therefore, the edge-based formulation reduces the mathematical dimensionality of the exact computation from $\mathcal{O}(N^4)$ down to $\mathcal{O}(N^3)$. 

It should be noted that while the theoretical complexity assumes that the number of
quadrature nodes increases linearly with the maximum order, in practice the adaptive Clenshaw--Curtis procedure typically converges with a nearly constant number of nodes.

Furthermore, for images containing contiguous shapes or objects, the number of active edges scales with the shape's perimeter, whereas the number of pixels scales with its area ($E \ll P$). Consequently, the gradient-weighted edge integration not only eliminates geometric discretization errors but, for typical structured images, provides infinite continuous precision at a vastly reduced computational burden compared to exact area summation.

\subsection{Generating an image from Zernike moments}
Given a set of complex, $Z_{n}^{m},$ or quaternion, $QZ_{n}^{m}$, Zernike moments, we reconstitute an image as follows. The size of the image is set to $N_x$ pixels along the $x$ dimension (width), and $N_y$ pixels along the $y$ dimension (height). Note that $N_x$ and $N_y$ are independent of the sizes of the image from which the Zernike moments were initially computed, allowing for resizing images. 

The center of the image is set at $C=(c_x,c_y) = (N_x/2, N_y/2)$. It is enclosed in a circle with radius $R = \sqrt{ c_x^2 + c_y^2}$.
A pixel $(i,j) \in [1, N_x]\times [1, N_y]$ is assigned the Cartesian coordinates:
\begin{eqnarray*}
x &=& \frac{(i - c_x -0.5)}{R} ,\quad y = \frac{(j - c_y -0.5)}{R} ,
\end{eqnarray*}
and the corresponding polar coordinates,
\begin{eqnarray*}
r_{x,y} &=& \sqrt{ x^2 + y^2}, \quad \theta_{x,y} = \tan^{-1} \left( \frac{y}{x} \right).
\end{eqnarray*}
The intensity $I(i,j)$ of this pixel is then set to
    \begin{equation}
        I(i,j) = \sum_{n=0}^{M} \sqrt{ \frac{n+1}{\pi} } \sum_{\substack{m=-n \\ n-|m| \text{ even}}}^{n} Z_n^m R_n^m(r_{x,y}) e^{im\theta_{x,y}}
     \label{eq:reconC}
    \end{equation}
    when the Zernike moments are complex, leading to a grayscale image, or to
\begin{equation}
        I(i,j) = \sum_{n=0}^{M} \sqrt{ \frac{n+1}{\pi} } \sum_{\substack{m=-n \\ n-|m| \text{ even}}}^{n} QZ_n^m R_n^m(r_{x,y}) e^{\mu m\theta_{x,y}}
        \label{eq:reconQ}
    \end{equation}
 when the Zernike moments are quaternion, leading to a color image. Note that in this equation $\mu = \frac{i+j+k}{\sqrt{3}}$. In both Equations~\ref{eq:reconC} and \ref{eq:reconQ}, $M$ is the reconstruction order.  As the Zernike moments are complete, the exact image is reconstructed when $M\rightarrow +\infty$.
 
 Note that computing the pixel intensities for the whole image is an embarrassingly parallel task.

\begin{figure*}[t]
\centering
\includegraphics[width=0.95\linewidth]{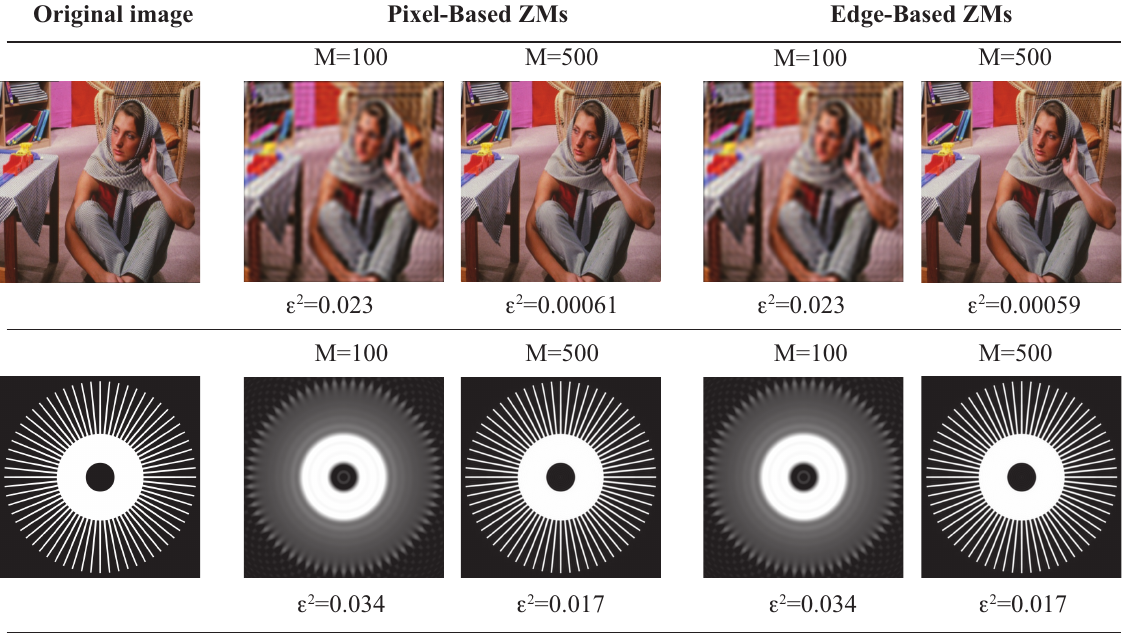}
\caption{Zernike-based reconstructed images and corresponding errors with different maximum order M of moments used.}
\label{fig:barbara}
\end{figure*}

\section{Numerical experiments}
\label{sec:exp}
The objective of the numerical experiments is twofold. First, we verify that the proposed edge-based formulation produces Zernike moments that are numerically equivalent
to those obtained with the classical pixel-based approach at moderate orders. 
Second, we demonstrate that the edge-based formulation remains stable at high orders, where spatial aliasing degrades the accuracy of the pixel-based computation.
The experiments are organized to evaluate these aspects separately. Reconstruction experiments assess the visual impact of exact integration, while shape
retrieval and classification experiments quantify the stability of high-order moments.

In all experiments, we compare the discrete, pixel-based approach briefly described in section \ref{subsec:img}, with the continuous, edge-based approach described in detail in the previous section. When appropriate, namely when the image can be captured as a set of contours, we apply the edge-based approach directly on the edges of those contours, as described above. All images are mapped inside the unit circle before the moment computation.

We use two new programs for the numerical experiments, Shape2Zernike, for computing moments, and Zernike2Shape, for regenerating an image given the Zernike moments. Both programs are written in C++. Both programs have been adapted to use the multiple cores available on the CPU on which they are run. For all experiments described below, we have used an Apple M4 Max processor with 16 cores (12 performance cores, 4 efficiency cores).

\subsection{Reconstructing color, grayscale, and binary images using pixel-based or edge-based QMs}

We considered first a color image. We used the standard Barbara image of size $(N_x, N_y)=(512, 512)$. Note that this image is drawn from A. Gersho's work and was first referenced in \cite{Shapiro:1993}. Quaternion Zernike moments were computed using both the discrete pixel-based and the continuous edge-based methods (see above) were computed up to order 500. The images reconstituded from these moments with different values of reconstruction order M are shown in
Fig. \ref{fig:barbara}. Let $I_0(i, j)$ and $I_r(i,j)$ be the (quaternion) intensities of the pixel $(i,j)$ in the  original image and in one of the reconstructed image, respectively. We compute the normalized mean square error  $\epsilon^2$ \cite{Revaud:2008} to measure the accuracy of the reconstructed images:
\begin{equation}
\epsilon^2 = \frac{ \sum_{i=1}^{N_x} \sum_{j=1}^{N_y} | I_0(i,j) - I_c(i,j)|^2} { \sum_{i=1}^{N_x} \sum_{j=1}^{N_y} | I_0(i,j)|^2}.
\label{eq:epsilon}
\end{equation}
The reconstruction errors are also given in Figure~\ref{fig:barbara}. The
results show that the reconstructed images are close to and very close to the original image for $M=100$ and $M=500$, respectively, for both the pixel-based and edge-based  Zernike moments. 
For natural color images of (relatively) high resolution, the edge-based and pixel-based methods therefore produce nearly indistinguishable reconstructions, demonstrating that the proposed method preserves the accuracy of the classical approach at moderate orders.

The second image is a $512\times 512$ black and white image of a wheel with a hole at its center, and 64 teeth at its boundary. The outer teeth radius is 0.95, the inner solid part has a radius of 0.45, and the hole has a radius of 0.15. The image was generated with the Python library Pillow \cite{Pillow}. Complex Zernike moments up to order N=500 were computed using the pixel-based and edge-based method, as described above, and images were reconstructed with $M=100$ and $M=500$. 

Unlike natural images, the gear contains a large number of sharp geometric features. Accurate reconstruction of these features requires stable estimation of high-order
moments. 
Consequently, this example provides a more stringent test of the numerical properties of the moment computation.
Results for the two methods look qualitatively and quantitatively (evaluated using $\epsilon^2$) the same, as observed in Figure~\ref{fig:barbara}. In both cases, the presence of a large number of thin teeth requires that a large number of moments be computed to generate a correct reconstruction.

\begin{figure*}[t]
\centering
\includegraphics[width=0.75\linewidth]{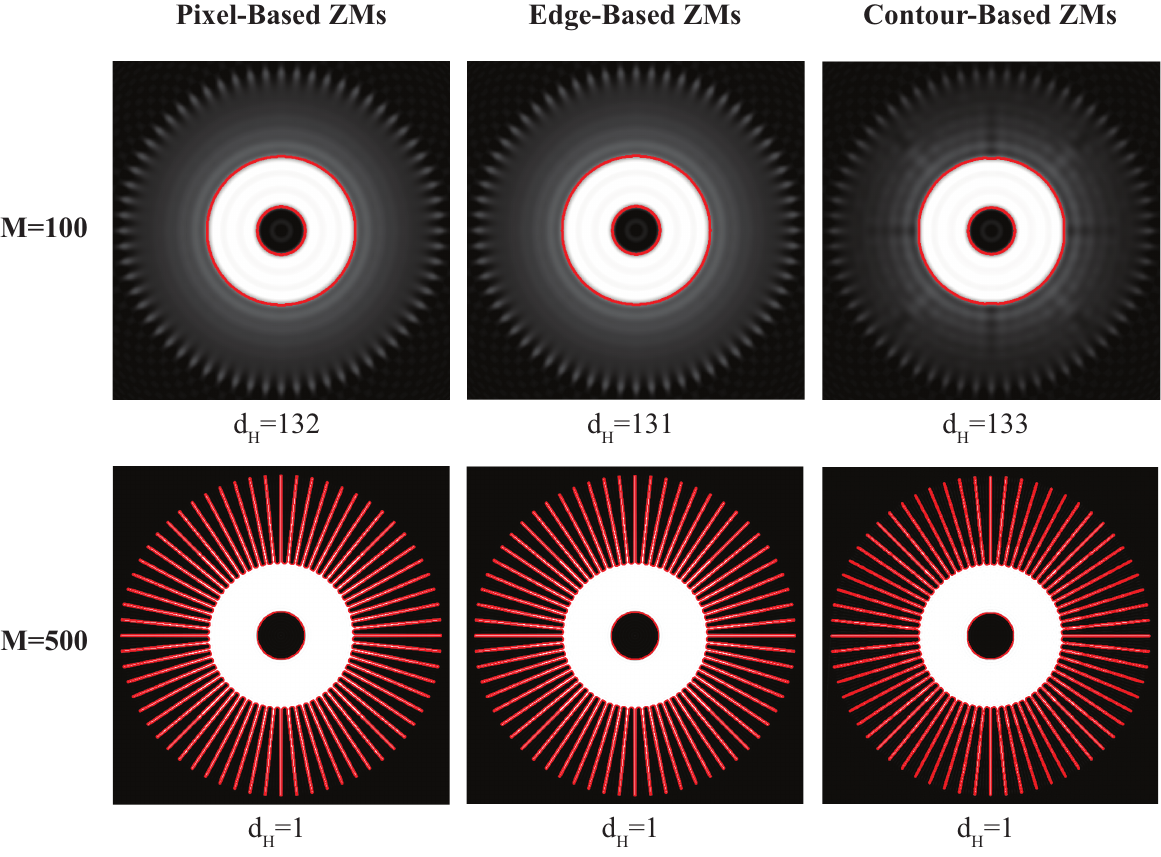}
\caption{Zernike-based reconstructed images (resolution 512x512) with the gear contours overlaid in red.  The Hausdorff distances $d_H$ between the gear contour in the original image (see Figure~\ref{fig:barbara}) and the contours in the reconstructed images are also  given. Results are shown for two levels of reconstructions, with $M=100$ and $M=500$ in the top and bottom row, respectively.}
\label{fig:contour}
\end{figure*}

The gear image in Figure~\ref{fig:barbara} is basically an annulus characterized with its outer (the teeth) and inner (the hole) boundaries, each defined as a simple polygonal curve. We used the Moore-Neighbor tracing algorithm modified by Jacob's stopping criteria as implemented in the Matlab function bwboundaries \cite{Gonzalez:2004} to identify those polygons. The outer polygon includes 15853 edges, while the inner hole is made of 285 edges.  Note that the Matlab function automatically assigns a different orientations for outer and inner boundaries. As described in the Method section, the complex Zernike moments for the gear can be computed directly from the two polygons. In Figure~\ref{fig:contour} we compare the reconstructed images based on the contour-based, pixel-based, and edge-based Zernike moments. The reconstruction errors of the contours are provided as Hausdorff distances $d_H$ between the vertices representing the outer and inner boundaries (computed using bwboundaries). Namely, for two annuli $A$ and $B$, the Hausdorff distance is given by:
\begin{equation}
d_H(A, B) = \sup\left\{\sup_{a\in A}\inf_{b\in B} d(a,b), \sup_{b\in B}\inf_{a\in A} d(a,b)\right\},
\end{equation}
where $d(a,b)$ is the Euclidean distance between $a$ and $b$. All the reconstructed images have the same resolution as the original gear image ($512\times 512$). Using Zernike moments up to $M=100$ is not enough to generate a faithful reconstruction, as the teeth of the gear are not recognized by the boundary detectors. In contrast, the quality of the reconstruction is highlighted by the fact that at $M=500$, the reconstructed and original contours are separated by 1 pixel only, for all three methods for computing the Zernike moments.

\subsection{Computing time}

We report computing time complexities as a function of $N_m$, the total number of moments. Note that $N_m$ is $\mathcal{O}(N^2)$, where $N$ is the maximum order of the Zernike moments (appreviated to ZMs from now).

The pixel based computation of ZMs  is expected to be $\mathcal{O}(N_xN_yN_m)$, where $N_x$ and $N_y$ are the width and height dimensions of the image. In contrast, and following the discussion on complexity in the previous section, the exact edge-based computation of ZMs is theoretically $\mathcal{O}(EN_m^{3/2})$, where $E$ is the number of edges. For an image, $E$ is  $\mathcal{O}(N_xN_y)$, and therefore the edge-based computation is $\mathcal{O}(N_xN_yN_m^{3/2})$. The exponent $3/2$ is associated with the fact that the exact computation of the line integral over the edge is $\mathcal{O}(N_m^{1/2})$.  Using an adaptive Clenshaw-Curtis quadrature, however, the integral may converge significantly faster.

\begin{figure}[h]
\centering
\includegraphics[width=0.95\linewidth]{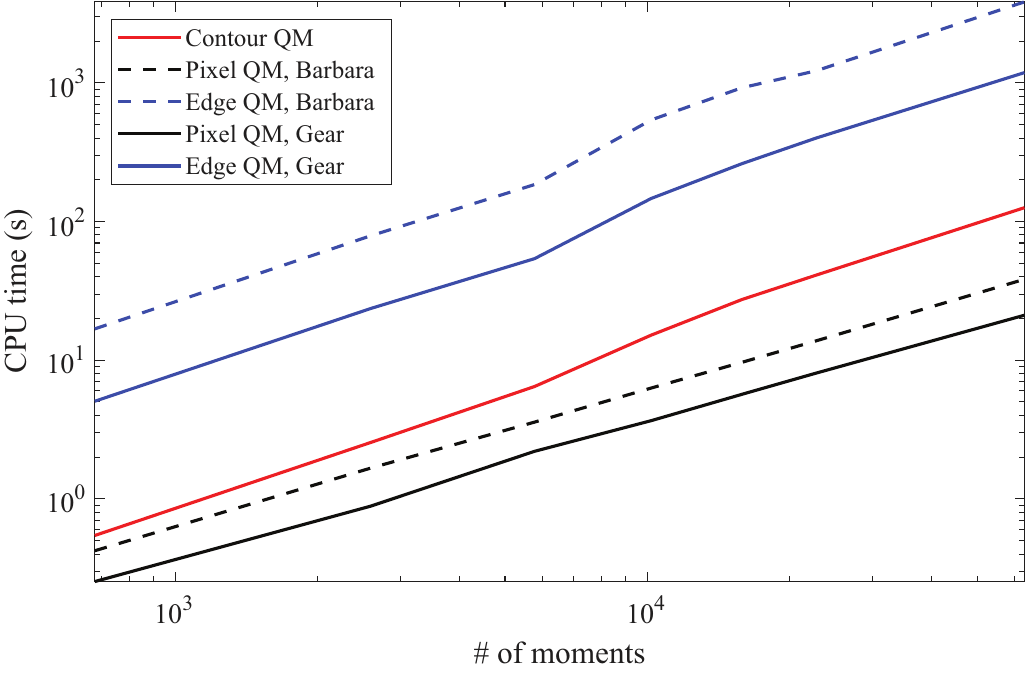}
\caption{Total CPU time for computing ZMs of either the Barbara image or the gear image as a function of the number of moments computed. Both images are 512x512. Computations are performed on 4 cores of a M4 Max processor.}
\label{fig:compute}
\end{figure}

\begin{figure*}[t]
\centering
\includegraphics[width=0.95\linewidth]{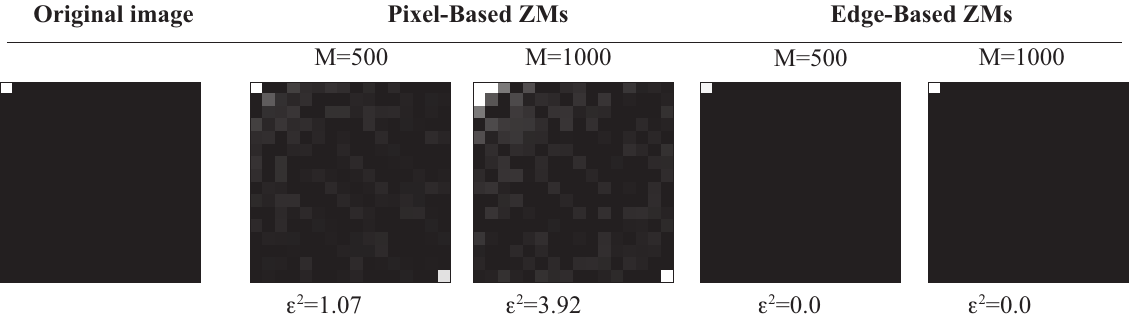}
\caption{\emph{The ghost pixel experiment}. Starting from a black picture with a single white pixel at its top left corner (left image), We compute Zernike moments up to order $M$ with two different methods, the standard pixel-based method and the edge-based method introduced in this paper. The corresponding reconstructed images with different maximum order M and their errors (computed using Equation~\ref{eq:epsilon}  are shown. Note the presence of a ghost pixel at the bottom right corner of the reconstructed images based on the pixel ZMs.}
\label{fig:onepixel}
\end{figure*}

In Figure~\ref{fig:compute}, we report the observed corresponding time complexities for two different images, the Barbara image and the gear image described above. Both input images include $(N_x,N_y)=(512,512)$ pixels.  The total CPU times are found to be linear with respect to $N_m$. As described above, this is expected for the pixel-based calculation. For the edge-based and contour-based computations, this apparent linearity is a result of the adaptive quadrature.
Although the worst-case complexity grows as $O(E N_m^{3/2})$, the adaptive quadrature procedure typically converges using only a small number of nodes (on average 17 nodes instead of the theoretical $N$ nodes required).
As a consequence, the observed runtime grows almost linearly with the number of moments.

There are, however, differences observed in Figure~\ref{fig:compute} that are worth commenting:
\begin{enumerate}
\item[i)] \emph{The pixel-based computations of ZMs are faster for the gear image than the Barbara image}. This is simply a consequence of the fact that the gear image is grayscale, requiring complex ZMs, while the Barbara image is color, for which we computed Quaternion ZMs.  Note that for complex Zernike moments, only moments for positive values of $m$ are actually computed, as $Z_n^{-m} = (Z_n^m)^{\ast}$. There are no similar simplifications for Quaternion moments, as quaternion multiplications are not symmetric.
\item[ii)] \emph{The edge-based computations of ZMs are faster for the gear image than the Barbara image}. The reason mentioned for pixel-based computation remains valid, but there is an additional factor to consider. An edge is considered if and only if there is a gradient in the intensities of the two adjacent pixels. For a sharp color image like Barbara, most edges have such a gradient. For the gear image, however, only $30\%$ of the total number of edges need to be considered.
\item[iii)] \emph{the edge-based computations are significantly slower than the pixel-based computation}. In the standard pixel-based computation, the integral is evaluated at a single point per pixel, typically its center. The edge-based computation, by contrast, relies on evaluating line integrals along the pixel boundaries. While every individual pixel is bounded by four edges, each internal edge is shared by two adjacent pixels. As a result, the grid effectively averages out to two unique edges per pixel. In addition, for an image of size 512 × 512, we observed that computing the line integral via adaptive Clenshaw-Curtis quadrature required an average of 17 evaluation points along each edge. Taking into account the effective two-to-one ratio of edges to pixels, the edge-based method requires evaluating at least 34 points per pixel.  In Figure~\ref{fig:compute}, we observe that on average, the edge computation is 37 and 74 times slower than the pixel computation for the gear image and the Barbara image, respectively.
\end{enumerate}
Finally, we note that each computation was performed using 4 cores. The corresponding wall times were approximately $25\%$ of the total CPU time, which is consistent with the fact that Algorithm \ref{algo:alg1} is trivially parallelizable. When trying 8 cores or more, we observed thermal underclocking.

\subsection{Discrete vs continuous computations of Zernike Moments}

The pixel-based method for computing Zernike moments is fundamentally discrete.
It equates each pixel with a single point at its center carrying its intensity. 
An image is then a collection of weighted Dirac functions located at the center of the pixels.
In contrast, the edge-based method considers a pixel as an actual square.
As this square is of uniform intensity, it suffices to integrate over its edges to obtain the correct values of its Zernike moments.

In the section above, we have shown that both methods compute Zernike moments that are accurate enough to generate faithful reconstructions of high resolution images.
As the computation of edge-based ZMs is significantly (one to two orders of magnitude) slower than the computation of pixel-based ZMs, it is understandable to wonder whether the former (introduced in this paper) is actually useful.
In the following, we answer this question positively, highlighting one inherent issue associated with discrete computation, namely spatial aliasing, and its consequence in the accuracy of the Zernike moments and the images they generate.

\subsubsection{The ghost pixel}

The limitations of the discrete approximation were first demonstrated using a simple 16x16 test image comprised of zero-intensity pixels, with the exception of a single white pixel (intensity 255) located at the top-left corner of the image.
As illustrated in Figure~\ref{fig:onepixel}, the image reconstructed using the pixel-based Zernike moments exhibits a distinct structural artifact - a ``ghost pixel" - at the antipodal bottom-right corner. 
This failure occurs because the Dirac delta approximation introduces severe spatial aliasing at high frequencies.
Briefly, Zernike polynomials possess strict origin symmetry. For any point at $(r, \theta)$, there is an exact opposite ``antipodal" point across the origin at $(r, \theta + \pi)$.
Even degree polynomials ($n = 0, 2, 4 \ldots$) have identical values at both sides, $V_n^m(r, \theta + \pi) = V_n^m(r, \theta)$
while odd degree polynomials ($n = 1, 3, 5 \ldots$) are perfectly inverted across the origin, $V_n^m(r, \theta + \pi) = -V_n^m(r, \theta)$.

To reconstruct a pixel at the top-left corner without adding a pixel at the bottom-right, the sum of all the odd polynomials must exactly match the sum of all the even polynomials.
When the high-frequency coefficients become corrupted by aliasing, the mathematical balance between the even and odd polynomials is destroyed. The most severe failure occurs at the point of maximum symmetry: the exact opposite corner, hence the presence of a ghost pixel.
Conversely, because the edge-based integration intrinsically suppresses frequencies smaller than the pixel width, it prevents aliasing; consequently, the ghost pixel artifact is entirely absent from the resulting reconstruction. 

\begin{figure}[t]
\centering
\includegraphics[width=0.95\linewidth]{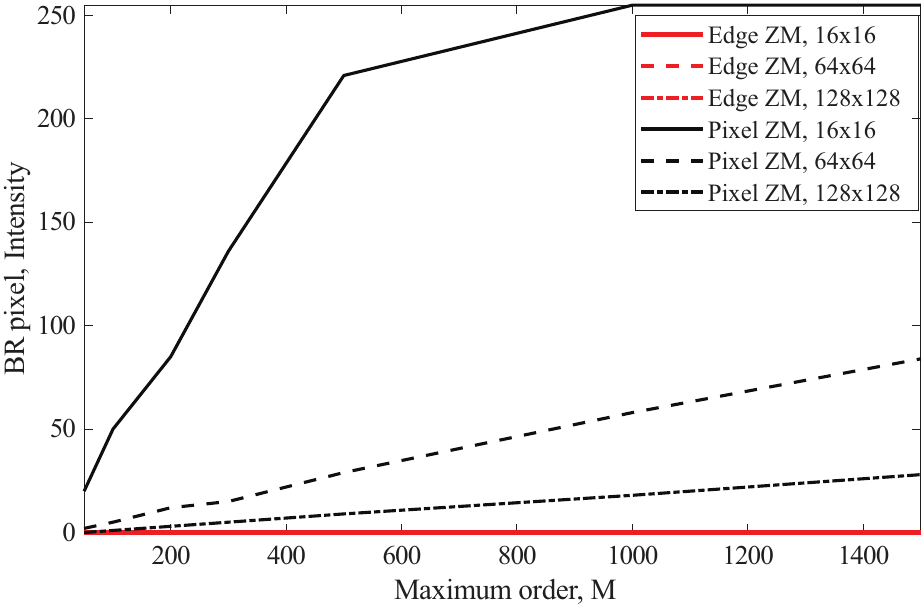}
\caption{The intensity of the bottom right (BR) pixel of the images reconstructed from Zernike moments computed with either the pixel-based or the edge-based method, as a function of the maximum order $M$ of the Zernike moments used for the reconstruction. As the images are treated as grayscale, a pixel intensity varies from 0 to 255.}
\label{fig:ghost}
\end{figure}

Figure~\ref{fig:onepixel} refers to an image of size $16 \times 16$. The ``ghost pixel" effect, however, is not limited to such low resolution images. We repeated the experiment with images of size $64\times 64$ and $256\times 256$. Those images were black with a single white pixel as the top left corner. In Figure~\ref{fig:ghost}, we report the intensity of the bottom right corner of the images reconstructed from Zernike moments computed with the pixel-based ZMs and edge-based ZMs. 
Consistent with the results from the $16\times 16$ image, there are no ghost pixels associated with the edge-based ZMs, while the intensities of the bottom right pixels in the pixel-based reconstructed images increase when the maximum order $M$ of reconstruction increases.

\begin{figure*}[t]
\centering
\includegraphics[width=0.7\linewidth]{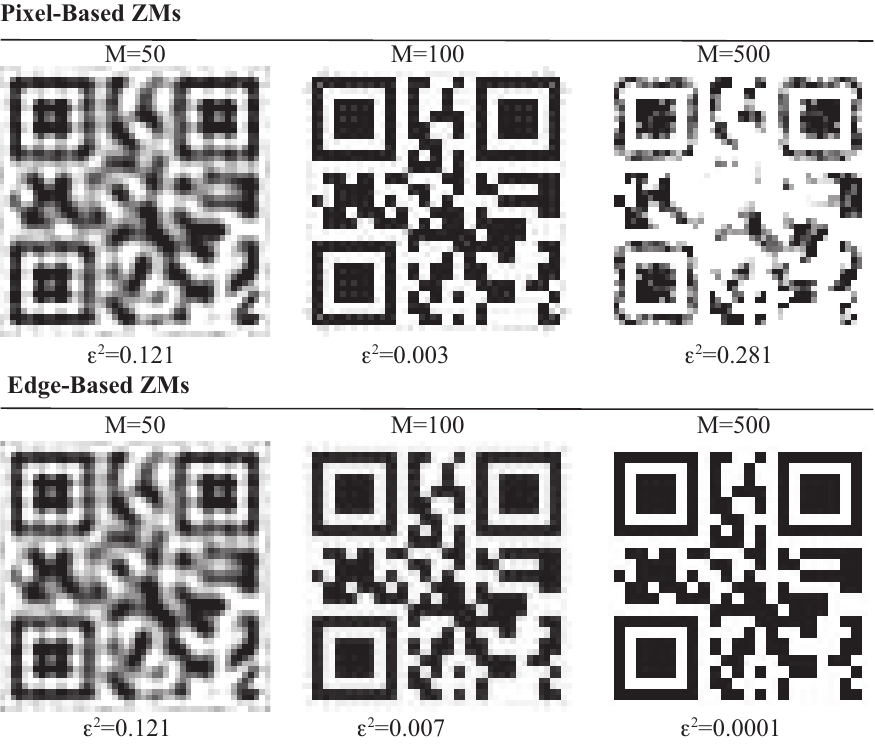}
\caption{Reconstructed images of the QR code representing the word ``Zernike" from pixel-based (top) and edge-based (bottom) complex Zernike moments, for different reconstruction orders $M$.  The reconstruction qualities, quantified by $\epsilon^2$ (Equation~\ref{eq:epsilon}) are provided below each image. Note that all those QR codes are recognized as ``Zernike", with the exception of the pixel-based reconstructed image with $M=500$, which is not identified as a QR code.}
\label{fig:QR}
\end{figure*}

\subsubsection{Corrupted QR codes}

Spatial aliasing may lead to bigger concerns than a ghost pixel for an image containing a single white pixel. 
Consider a QR code: spatial aliasing will lead to degradation of its reconstruction from Zernike moments for high frequencies, i.e., large order $M$. 
We illustrate this behavior in Figure~\ref{fig:QR}. A $48\times 48$ image of the QR code representing the word ``Zernike" was generated using the library qrcode from Python. Pixel-based, and edge-based complex Zernike moments were generated from this image, up to order $N=500$. 
The images reconstructed from those Zernike moments at order $M=50, 100$, and $500$ are shown in Figure~\ref{fig:QR}. 
For both types of moments, the reconstructed image at $M=50$ is somewhat blurry, but still recognizable as the QR code ``Zernike". 
At $M=100$, the pixel-based and the edge-based reconstructed  images look significantly sharper, with good $\epsilon^2$ values. 
The quality of the reconstruction, however, changes drastically for a higher reconstruction order $M=500$. 
While the quality of the reconstruction keeps improving for the edge-based Zernike moments, the pixel-based reconstruction is degraded, leading to an image that is no more recognized as a QR code. This degradation is the result of spatial aliasing at high frequencies.

\subsubsection{Spatial aliasing problems in image classification}

\begin{figure*}[t]
\centering
\includegraphics[width=0.9\linewidth]{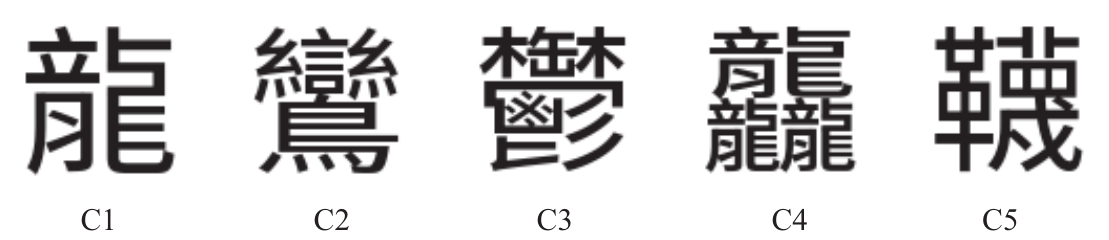}
\caption{The five Chinese characters, which we refer to as C1, C2, C3, C4, and C5. These characters were not selected for their semantic meanings (these are traditional, often archaic ideograms), but because of their complexities (C2 and C4 for example are among the most stroke-heavy characters in Chinese).}
\label{fig:chinese}
\end{figure*}

The two subsections above highlighted image reconstruction problems associated with spatial aliasing inherent to computing Zernike moments using a point-based representation of an image. Here we show that this problem extends to other applications of Zernike moments, more specifically image classification. We compare classifications obtained with pixel-based, and edge-based Zernike moments.

We considered a set of 5 Chinese characters, illustrated in Figure~\ref{fig:chinese}.
We generated 50 samples for each of those characters. 
A sample is a grayscale image of size $48\times 48$, representing a single character.
It is the result of a random rotation (in the range $[-15^{\circ}, +15^{\circ}]$) followed by a 2D affine transformation that randomly scales (in the range $[0.9, 1.1]$) and shears (with a shear factor in the range $[-0.08, 0.08]$) of that character.
Complex Zernike moments are computed for each sample, using either the pixel-based, or edge-based method. 
As the Zernike moment $Z_n^m$ is not invariant to rotation, translation, or scale, we convert it to a Zernike Invariant (ZI) \cite{Khotanzad:1990}, $I_n^m$ using:
\begin{equation*}
I_n^m = \frac{|Z_n^m|}{|Z_0^0|},
\end{equation*}
where $|\cdot|$ refers to the modulus. Note that $Z_n^m$ is a complex number while $I_n^m$ is a real number.
The distance $d(i,j)$ between two samples $i$ and $j$ is then computed as the Euclidean norm between their Zernike invariants:
\begin{equation*}
d(i,j) = \sqrt{\sum_{n=0}^{M} \sum_{m} (I_n^m(i) - I_n^m(j))^2},
\end{equation*}
where $m \in [0,n]$ with the condition that $n-m$ is even and $M$ is the maximum order considered. We only consider the positive values of $m$ as for complex Zernike moments, $I_n^{-m}=I_n^m$.

\begin{figure*}
        \centering
        \includegraphics[width=0.75 \linewidth]{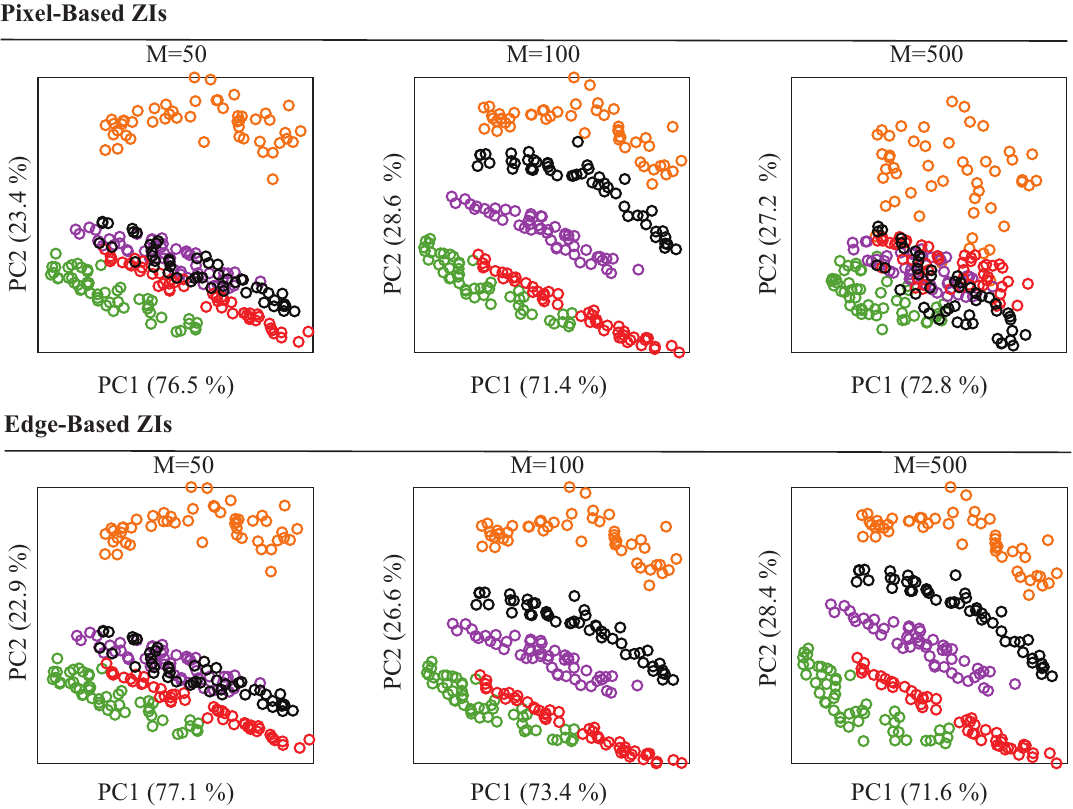}
        \caption{\emph{The MDS embeddings of the samples of the 5 Chinese characters shown in Figure~\ref{fig:chinese}}. The distances between the samples are defined as the distances between their Zernike invariants,  computed either based on a discrete pixel-based representation of the images (top), or on a continuous edge-based representation of the images (bottom), with different maximum orders, M. Characters C1, C2, C3, C4, and C5 are represented in orange, green, purple, red, and black, respectively. The explained variance of each principal component is provided in parenthesis.}
        \label{fig:mds}
    \end{figure*}
    
We computed the distance matrices over the 250 samples (50 samples $\times$ 5 characters) for the pixel-based ZIs and the edge-based ZIs, with maximum orders M=50, 100, 200, 300, and 500. We projected the information contained in these matrices on a two dimensional space using multidimensional scaling (MDS \cite{Gower:1966, Borg:2005}). Results are shown in Figure~\ref{fig:mds}. 
The Zernike invariants perform well in distinguishing the 5 Chinese characters. For the edge-based Zernike invariants, character C1 (in orange) and to some extent character C2 (in green) are well separated from the other characters for M = 50. The five characters appear to be fully separated for M=100, and this result remains stable for M=500. For the pixel-based Zernike invariants, results for M=50 and M=100 are very similar to the corresponding results for the edge-based Zernike invariants. However, results for M=500 show that the Zernike invariants are not stable at high orders, leading to poor separation of the characters.

Figure~\ref{fig:mds}  provides an illustration of the impact of spatial aliasing on classification successes of Zernike invariants; we still need to quantify how important the loss of accuracy is for the pixel-based Zernike invariants . We used two complementary approaches for that purpose, namely a receiver operating characteristic (ROC) analysis \cite{Metz:1978} and a classification test.

The five Chinese characters in Figure~\ref{fig:chinese} serve as the standard. A pair of samples is defined as similar, or ``positive", if they correspond to the same character, and ``negative”"otherwise. All pairs of samples are compared using one of the ZI distances defined above. For varying thresholds of this measure, all pairs below the threshold are assumed positive, and all above it are negative. The pairs that agree with the standard are called true positives (TP), while those that do not are false positives (FP). ROC analysis compares the rate of TP as a function of the rate of FP; it is scored with the area below the corresponding curve, AUC \cite{Hanley:1982}.

We extended the ROC analysis to the problem of classification. Classification accuracy was estimated directly from the distance matrix using repeated random hold-out validation. For each repetition, samples from each class were randomly divided into equally sized reference and test sets. A test sample was assigned to the class whose reference samples exhibited the smallest average distance to that sample. Classification performance was summarized by the confusion matrix and by the overall accuracy, defined as the proportion of correctly classified test samples. Reported accuracies are averaged over 10000 independent random partitions.

Table \ref{table:class} summarizes the results of the classification and ROC analyses for both pixel-based and edge-based Zernike invariants. For both representations, the classification accuracy and the AUC increase steadily as the maximum order of the invariants increases from $M=10$ to $M=100$. This behavior is expected, as higher-order invariants capture progressively finer details of the character shapes and therefore provide a more discriminative representation of the images.

Beyond $M=100$, however, the two approaches exhibit markedly different behavior. For the edge-based invariants, both the classification accuracy and the AUC remain essentially constant, reaching values of $97\%$ and $0.84$, respectively. This indicates that increasing the maximum order beyond $100$ neither improves nor degrades the discriminative power of the descriptors, suggesting that the edge-based computation remains numerically stable at high orders.

In contrast, the performance of the pixel-based invariants decreases for $M>100$. The classification accuracy drops from $96.9\%$ at $M=100$ to $90.1\%$ at $M=500$, while the AUC decreases from $0.83$ to $0.78$. Since higher-order invariants should not reduce the amount of shape information available, this loss of performance is most likely the result of numerical errors in the computation of the pixel-based invariants, associated with spatial aliasing. These errors accumulate at high orders and progressively obscure the discriminatory information contained in the descriptors. This interpretation is consistent with the MDS projections shown in Figure~\ref{fig:mds}, where the separation between character classes deteriorates for the pixel-based invariants at large values of $M$, whereas it remains essentially unchanged for the edge-based invariants.

\begin{table*}[t]
	\begin{threeparttable}
        \caption{Classifying images of Chinese characters using Zernike invariants}
        \label{table:class}
                \begin{tabular}{l l c c c c c c c c c}
                \cline{3-11}
 & &\multicolumn{5}{c}{Maximum order M} \cr
\hline
  Invariants & Test & 10 & 20 & 30 & 40 & 50 & 100 & 200 & 300 & 500\\
\hline
Pixel \tnote{a} & Accuracy \tnote{c} & 70.5 (4.7) & 82.5 (4.2) & 86.8 (4.0) & 94.2 (2.9) & 94.2 (2.7)  & 96.9 (2.0)  & 95.6 (2.5) & 91.1 (3.2) & 90.1 (3.2) \\
Pixel & AUC \tnote{d}  & 0.70 & 0.74 & 0.75 & 0.80 & 0.81 & 0.83 & 0.82 & 0.79  & 0.78 \\
Edge  \tnote{b} & Accuracy & 70.5 (4.7) & 82.6 (4.2) & 86.6 (4.0) & 93.9 (3.0) & 94.0 (2.8) & 96.0 (2.2) & 97.1 (2.0) & 97.1 (2.1) & 97.3 (2.0) \\
Edge & AUC & 0.70 & 0.74 & 0.76 & 0.80 & 0.81 & 0.83 & 0.84 & 0.84  & 0.84 \\
\hline
 \end{tabular}
 \begin{tablenotes}
 \item [a)] {\small Pixel-based Zernike invariants.}
  \item [b)] {\small Edge-based Zernike invariants.}
   \item [c)] {\small Accuracy of classification experiments, in $[0, 100\%]$, averaged over 10000 cross-validation experiments; the standard deviation is provided in parenthesis;  the higher, the better.}
    \item [d)] {\small Area under the curve for ROC experiments, in $[0, 1]$; the higher, the better.}
 \end{tablenotes}
 \end{threeparttable}
\end{table*}

\section{Conclusion}
\label{sec:conclusion}
This paper introduces a new edge-based formulation for the computation of two-dimensional Zernike moments and invariants from digital images. 
Unlike conventional pixel-based approaches, which approximate the continuous image integrals by discrete sums over pixel centers, the proposed method represents image boundaries as polygonal contours and evaluates the corresponding moment integrals exactly along their edges. 
The resulting formulation eliminates the spatial discretization errors inherent to pixel-based methods while preserving the desirable properties of Zernike moments, including orthogonality, compactness, and rotational invariance.

A key contribution of this work is the demonstration that the edge contribution to a Zernike moment can be transformed into the integral of a polynomial function. 
This property allows the use of Clenshaw--Curtis quadrature as an exact integration scheme rather than as a numerical approximation. 
Consequently, the computation of the moments is no longer limited by the spatial aliasing effects that affect traditional pixel-based formulations, particularly at high orders.

The proposed framework applies equally well to binary, grayscale, and color images. For grayscale and color images, the image is first decomposed into a collection of polygonal regions associated with constant intensity values, after which the moment computation reduces to the evaluation of contour integrals along the corresponding region boundaries. 
This formulation generalizes to any 2D shape defined from its boundaries, both internal and external.

The numerical experiments demonstrate that the edge-based formulation reproduces the results of the classical pixel-based approach at low and moderate orders while exhibiting substantially improved stability at high orders. 
Reconstruction experiments confirm that the proposed method preserves image fidelity, whereas classification experiments on a dataset of Chinese characters show that high-order edge-based invariants retain and even improve their discriminative power even when the corresponding pixel-based invariants begin to deteriorate. 
In particular, for a maximum order of 500, the edge-based invariants achieve a classification accuracy of 97.3\%, compared with 90.1\% for the pixel-based invariants, highlighting the practical impact of eliminating discretization errors from the moment computation.

Several extensions of this work may be considered. 
The edge-based integration framework could be generalized to other families of orthogonal moments, including Legendre, Chebyshev, and Fourier--Mellin descriptors. The independence of edge contributions also makes the method particularly well suited to parallel implementations on modern multicore and GPU architectures. 
Finally, the extension of the present approach to three-dimensional Zernike moments computed from polygonal surface meshes represents a promising direction for future research.

\section*{Code availability}
\label{sec:code}
The C++ source codes and the scripts to generate the datasets used in this study will be made publicly available on GitHub upon publication.

\section*{Acknowledgments}

This work was initiated while P.K. was an international visitor of the Sydney
Mathematical Research Institute (SMRI), University of Sydney, Australia. He is grateful for their financial support and hospitality during his stay.


{
\appendix[Proofs of the recurrences for the radial polynomials $R_n^m(r)$ and their integrals.]
\label{sec:Appendix}
Many recurrences have been proposed for the Zernike radial polynomial functions $R_n^m(r)$ (see for example \cite{Kintner:1976, Noll:1976, Prata:1989, Chong:2003, Shakibaei:2013, Deng:2016, Andersen:2018}), sometimes adapted to specific software libraries and/or to special processor architecture \cite{Qin:2012, Suhng:2024, Elmacioglu:2025}. The purpose of these recurrences is to improve stability for larger order $n$, as well as computing speed. The recurrences we propose here are not novel (except the one associated with computing $Q_n^m(r)$). We do provide full proofs of all three recurrences in Proposition 1 from the main body of this paper.

The real-valued radial polynomial $R_n^m(r)$ is defined by:
\begin{equation}
R_n^{|m|}(r) = \sum_{k=0}^{(n-|m|)/2} (-1)^k F(n,m,k) r^{\,n-2k},
\end{equation}
with
\begin{equation*}
F(n,m,k) = \frac{(n-k)!}{k!\left(\frac{n+|m|}{2}-k\right)!\left(\frac{n-|m|}{2}-k\right)!},
\end{equation*}
where $n$ is a non-negative integer, $m$ is an integer satisfying $n-m$ is even and $|m| \le n$. 
For the sake of clarity, we will assume $m \ge 0$ and remove the absolute values. There are many ways to rewrite this polynomial. For example, we can rewrite the ratios of factorials as products of binomials,
\begin{equation}
R_n^{m}(r)
= \sum_{k=0}^{(n-m)/2} (-1)^k \binom{n-k}{k} \binom{n-2k} {\frac{n-m}{2}-k} r^{n-2k},
\end{equation}
revealing that the coefficients in the polynomial expansion of $R_n^m(r)$ are integers. To reveal useful recurrences, it is noticed that $R_n^m(r)$ are special cases of Jacobi polynomials $P_{n}^{(\alpha, \beta)}$:
\begin{equation}
R_n^{m}(r) = (-1)^{(n-m)/2} r^m P_{(n-m)/2}^{(m,0)} (1-2r^2),
\end{equation}
or can be expressed as an integral of a Chebyshev polynomial of the second type $U_n(x)$ \cite{Janssen:2007}:
\begin{equation}
R_n^{m}(r) = \frac{1}{2\pi} \int_{0}^{2\pi} U_n(r \cos(\theta)) \cos(m\theta) d\theta,
\end{equation}
or can be expressed as an integral of Bessel functions of the first kind \cite{Noll:1976}:
\begin{equation}
R_n^{m}(r) = (-1)^{(n-m)/2} \int_{0}^{+\infty} J_{n+1}(x) J_{m}(rx) dx.
\label{eqn:b0}
\end{equation}
We will use the latter, as it gives us access to two fundamental recurrences:
\begin{eqnarray}
J_{n-1}(x) + J_{n+1}(x) &=& \frac{2n}{x} J_n(x) \label{eqn:b1}, \\
J_{n-1}(x) - J_{n+1}(x) &=& 2 J'_n(x) \label{eqn:b2}.
\end{eqnarray}

We will assume $n > 0$, $1 \le m \le n-2$ and $n-m$ even. To improve readability, we define:
\begin{eqnarray*}
A_n^m = (-1)^{(n-m)/2}.
\end{eqnarray*}.

Using Equation~\ref{eqn:b2}, $R_n^m(r)$ can be rewritten as:
\begin{eqnarray*}
R_n^{m}(r) &=& A_n^m \int_{0}^{+\infty} J_{n+1}(x) J_{m}(rx) dx \\
&=& A_n^m \int_{0}^{+\infty} \left(J_{n-1}(x) - 2 J'_n(x) \right) J_{m}(rx) dx \\
&=& - R_{n-2}^m(r) - 2 A_n^m \int_{0}^{+\infty} J'_n(x) J_{m}(rx) dx.
\end{eqnarray*}
After integration by parts of the second term on the right of this equation:
\begin{eqnarray*}
R_n^{m}(r) &=& - R_{n-2}^m(r) - 2 (A_n^m \left[ J_n(x) J_m(rx) \right]_{0}^{+\infty} \\
&& + 2 A_n^m \int_{0}^{+\infty} J_n(x) \frac{\partial}{\partial x}(J_{m}(rx)) dx \\
&=& - R_{n-2}^m(r) + 2 A_n^m \int_{0}^{+\infty} J_n(x) J'_{m}(rx) r dx.
\end{eqnarray*}
We used the fact that $J_n(0)=0$ and $\lim_{x\rightarrow +\infty} J_n(x) = 0$ for $n >0$. Now we use Equation~\ref{eqn:b2} again on the derivative $2 J'_m(rx) = J_{m-1}(rx) - J_{m+1}(rx)$:
\begin{eqnarray}
R_n^{m}(r) &=& - R_{n-2}^m(r) + A_n^m r \int_{0}^{+\infty} J_n(x) J_{m-1}(rx) dx \nonumber \\ 
&& - A_n^m r \int_{0}^{+\infty} J_n(x) J_{m+1}(rx) dx \nonumber \\
&=& - R_{n-2}^m(r) + r (R_{n-1}^{m-1}(r) + R_{n-1}^{m+1}(r)),
\label{eqn:a2eq3}
\end{eqnarray}
which concludes the proof of the recurrence of Equation~\ref{eqn:rec1} in the main text.

To prove Equation~\ref{eqn:rec2}, we start again with Equation~\ref{eqn:b0} and differentiate both sides with respect to $r$. Noting that $\frac{\partial}{\partial r} J_m(rx) = x J'_m(rx)$:
\begin{eqnarray*}
\frac{dR_n^{m}}{dr}(r) &=& A_n^m \int_{0}^{+\infty} J_{n+1}(x) x J'_{m}(rx) dx.
\end{eqnarray*}
Based on Equation~\ref{eqn:b1}, $xJ_{n+1}(x) = 2nJ_n(x) - xJ_{n-1}(x)$. Therefore,
\begin{eqnarray*}
\frac{dR_n^{m}}{dr}(r) &=& -A_n^m \int_{0}^{+\infty} J_{n-1}(x) x J'_{m}(rx) dx   \\
&& +  2n A_n^m \int_{0}^{+\infty} J_{n}(x) J'_{m}(rx) dx.
\end{eqnarray*}
Because the first integral internally reconstructs the $r$-derivative, we have:
\begin{eqnarray}
\frac{dR_n^{m}}{dr}(r) &=& \frac{dR_{n-2}^{m}}{dr}(r) + 2n A_n^m \int_{0}^{+\infty} J_{n}(x) J'_{m}(rx) dx. \nonumber \\
\label{eqn:a2eq4}
\end{eqnarray}
Using Equation~\ref{eqn:b2}:
\begin{equation}
2 J'_{m}(rx) = J_{m-1}(rx) - J_{m+1}(rx).
\label{eqn:a2eq5}
\end{equation}
Replacing Equation~\ref{eqn:a2eq5} into Equation~\ref{eqn:a2eq4}, we get:
\begin{eqnarray}
\frac{dR_n^{m}}{dr}(r) &=& \frac{dR_{n-2}^{m}}{dr}(r) - n A_n^m \int_{0}^{+\infty} J_{n}(x) J_{m-1}(rx) dx \nonumber \\
&& + nA_n^m \int_{0}^{+\infty} J_{n}(x) J_{m+1}(rx) dx \nonumber \\
&=& \frac{dR_{n-2}^{m}}{dr}(r) + n \left( R_{n-1}^{m-1}(r) + R_{n-1}^{m+1}(r) \right).
\label{eqn:a2eq6}
\end{eqnarray}
Integrating Equation~\ref{eqn:a2eq6} over $r$ from $0$ to $r$, we get:
\begin{eqnarray*}
R_n^m(r) = R_{n-2}^m(r) + n \left( S_{n-1}^{m-1}(r) + S_{n-1}^{m+1}(r) \right).
\end{eqnarray*}
Shifting $n\rightarrow n+1$ and $m \rightarrow m+1$, we get:
\begin{eqnarray}
R_{n+1}^{m+1}(r) = R_{n-1}^{m+1}(r) + (n+1) \left( S_{n}^{m}(r) + S_{n}^{m+2}(r) \right),
\label{eqn:a2eq7}
\end{eqnarray}
which concludes the proof of Equation~\ref{eqn:rec2} (after isolating $S_n^m$ and shifting $m \to m-2$). A similar proof was given by Noll \cite{Noll:1976}.

Finally, shifting $n\rightarrow n+1$ and $m \rightarrow m+1$ in Equation~\ref{eqn:a2eq3}, we get:
\begin{eqnarray*}
R_{n+1}^{m+1}(r) = - R_{n-1}^{m+1}(r) + r \left( R_{n}^{m}(r) + R_{n}^{m+2}(r) \right).
\end{eqnarray*}
After integration over $r$, recognizing that $\int_0^r R_n^m(t) dt = S_n^m(r)$ and $\int_0^r t R_n^m(t) dt = Q_n^m(r)$, we get:
\begin{eqnarray*}
S_{n+1}^{m+1}(r) = -S_{n-1}^{m+1}(r) + Q_{n}^{m}(r) + Q_{n}^{m+2}(r),
\end{eqnarray*}
which concludes the proof of Equation~\ref{eqn:rec3} (after shifting $m \to m-2$).
}


\bibliographystyle{IEEEtran}
\bibliography{IEEEabrv, zernike}

@String{PSFB =            {Proteins: Struct.\ Func.\ Bioinfo.}}

@String{BJ =              {Biophys.\ J.}}

@article{Kihara:2009,
	author={V. Venkatraman and L. Sael and D. Kihara},
	title={Potential for protein surface shape analysis using spherical harmonics and {3D Z}ernike descriptors},
	journal={Cell.\ Biochem.\ Biophys.},
	volume={54},
	pages={23--32},
	year={2009},
}

@article{Teague:1980,
	author={M. Teague},
	title={Image analysis via the general theory of moments},
	journal={J.\ Opt.\ Soc.\ Amer.},
	volume={70},
	pages={920--930},
	year={1980},
}

@article{Hu:1962,
	author={M. Hu},
	title={Visual pattern recognition by moment invariants},
	journal={{IRE} Trans. Infor.\ Theory},
	volume={8},
	pages={179--187},
	year={1962},
}

@inproceedings{Novotni:2003,
	author={M. Novotni and R. Klein},
	title={{3D Z}ernike descriptors for content based shape retrieval},
	booktitle={Proc.\ {ACM} symposium on solid and physical modeling},
	pages={216--225},
	year={2003},
}

@article{Novotni:2004,
	author={M. Novotni and R. Klein},
	title={Shape retrieval using {3D Zernike} descriptors},
	journal={Computer Aided Design},
	volume={36},
	pages={1047--1062},
	year={2004},
}

@article{Venkatraman:2009,
author = {Venkatraman,V. and Yang,Y. and Sael,L. and Kihara,D.},
title = {Protein-protein docking using region-based {3D Zernike} descriptors},
journal = {BMC Bioinformatics},
year = {2009},
volume = {10},
pages = {407},
}

@article{Sael:2008,
  author = {Lee Sael and Bin Li and David La and Yi Fang and Karthik Ramani and Raif Rustamov and Daisuke Kihara},
  title = {Fast protein tertiary structure retrieval based on global surface shape similarity},
  journal = PSFB,
  year = {2008},
  volume = {72},
  pages = {1259--1273},
}

@article{Zernike:1934,
	author = {F. Zernike},
	title = {Beugungstheorie des {S}chneidenver-fahrens und seiner verbesserten {F}orm, der {P}hasenkontrastmethode},
	journal = {Physica},
	volume = {1},
	number = {7},
	pages = {689--704},
	year = {1934},
}

@article{Kintner:1976,
	title={On the Mathematical Properties of the {Zernike} Polynomials},
	author={Kintner, Eric C},
	journal={Optica Acta},
	volume={23},
	number={8},
	pages={679--680},
	year={1976}
}

@article{Prata:1989,
	title={Algorithm for computation of {Zernike} polynomials expansion coefficients},
	author={Prata, Aluizio and Rusch, WVT},
	journal={Applied Optics},
	volume={28},
	pages={749--754},
	year={1989},
}

@article{Noll:1976,
	title={Zernike polynomials and atmospheric turbulence},
	author={Noll, Robert J},
	journal={Journal of the Optical Society of America},
	volume={66},
	pages={207--211},
	year={1976},
}

@article{Chen:2015,
	title={Color image analysis by quaternion-type moments},
	author={Chen, Beijing and Shu, Huazhong and Coatrieux, Gouenou and Chen, Gang and Sun, Xingming and Coatrieux, Jean Louis},
	journal={Journal of mathematical imaging and vision},
	volume={51},
	pages={124--144},
	year={2015},
}

@article{Chong:2003,
	title={A comparative analysis of algorithms for fast computation of {Z}ernike moments},
	author={Chong, Chee-Way and Raveendran, P and Mukundan, Ramakrishnan},
	journal={Pattern Recognition},
	volume={36},
	pages={731--742},
	year={2003},
}

@article{Deng:2016,
	title={Stable, fast computation of high--order {Zernike} moments using a recursive method},
	author={Deng, An-Wen and Wei, Chia-Hung and Gwo, Chih-Ying},
	journal={Pattern Recognition},
	volume={56},
	pages={16--25},
	year={2016},
}

@article{Elmacioglu:2025,
	title={{ZERNIPAX}: A fast and accurate {Zernike} polynomial calculator in {Python}},
	author={Elmacioglu, Yigit Gunsur and Conlin, Rory and Dudt, Daniel W and Panici, Dario and Kolemen, Egemen},
	journal={Applied Mathematics and Computation},
	volume={505},
	pages={129534},
	year={2025},
}

@inproceedings{Suhng:2024,
	title={Research on Fast {Zernike} Moment Computation in a Multiple-GPU Environment},
	author={Suhng, ByuhngMunn and Lee, Wangheon},
	booktitle={2024 24th International Conference on Control, Automation and Systems (ICCAS)},
	pages={459--464},
	year={2024},
	organization={IEEE}
}

@article{Qin:2012,
	title={A parallel recurrence method for the fast computation of {Zernike} moments},
	author={Qin, Huafeng and Qin, Lan and Xue, Lian and Yu, Chengbo},
	journal={Applied Mathematics and Computation},
	volume={219},
	pages={1549--1561},
	year={2012},
}

@article{Andersen:2018,
	title={Efficient and robust recurrence relations for the {Zernike} circle polynomials and their derivatives in {Cartesian} coordinates},
	author={Andersen, Torben B},
	journal={Optics Express},
	volume={26},
	pages={18878--18896},
	year={2018},
}

@book{Born:1999,
	author    = {Born, Max and Wolf, Emil},
	title     = {Principles of Optics},
	edition   = {7th},
	publisher = {Cambridge University Press},
	address   = {Cambridge},
	year      = {1999},
	doi       = {10.1017/CBO9781139644181}
}

@article{Mahajan:1994,
	author    = {Mahajan, Virendra N.},
	title     = {Zernike Circle Polynomials and Optical Aberrations of Systems with Circular Pupils},
	journal   = {Applied Optics},
	volume    = {33},
	pages     = {8121--8124},
	year      = {1994},
	doi       = {10.1364/AO.33.008121}
}

@article{Thibos:2002,
	author    = {Thibos, Larry N. and Applegate, Raymond A. and Schwiegerling, James T.
               and Webb, Robert},
	title     = {Standards for Reporting the Optical Aberrations of Eyes},
	journal   = {Journal of Refractive Surgery},
	volume    = {18},
	number    = {5},
	pages     = {S652--S660},
	year      = {2002},
	doi       = {10.3928/1081-597X-20020901-30}
}

@article{Khotanzad:1990,
	author    = {Khotanzad, Alireza and Hong, Yaw Hua},
	title     = {Invariant Image Recognition by {Zernike} Moments},
	journal   = {{IEEE} Transactions on Pattern Analysis and Machine Intelligence},
	volume    = {12},
	number    = {5},
	pages     = {489--497},
	year      = {1990},
	doi       = {10.1109/34.55109}
}

@article{Kan:2002,
	author = {Chao Kan and Mandyam D. Srinath},
	title = {Invariant character recognition with {Zernike} and orthogonal {Fourier--Mellin} moments},
	journal = {Pattern Recognition},
	volume = {35},
	pages = {143-154},
	year = {2002},
}

@article{Tahmasbi:2011,
	author    = {Tahmasbi, Amir and Saki, Fatemeh and Shokouhi, Shahriar Baradaran},
	title     = {Classification of Benign and Malignant Masses Based on {Zernike} Moments},
	journal   = {Computers in Biology and Medicine},
	volume    = {41},
	number    = {8},
	pages     = {726--735},
	year      = {2011},
	doi       = {10.1016/j.compbiomed.2011.06.009}
}

@article{Mudigonda:2000,
	author    = {Mudigonda, Nanda R. and Rangayyan, Rangaraj M. and Desautels, J. E. Leo},
	title     = {Detection of Breast Masses in Mammograms by Density Slicing and
               Texture Flow-Field Analysis},
	journal   = {IEEE Transactions on Medical Imaging},
	volume    = {20},
	number    = {12},
	pages     = {1215--1227},
	year      = {2001},
	doi       = {10.1109/42.974917}
}

@article{Banach:2024,
	author    = {Mateusz Banach},
	title     = {Structural Outlier Detection and {Zernike--Canterakis} Moments for Molecular Surface Meshes---Fast Implementation in {Python}},
	journal   = {Molecules},
	volume    = {29},
	pages     = {52},
	year      = {2024},
	doi       = {10.3390/molecules29010052},
}

@article{Boland:2001,
	author    = {Boland, Michael V. and Murphy, Robert F.},
	title     = {A Neural Network Classifier Capable of Recognizing the Patterns of All
               Major Subcellular Structures in Fluorescence Microscope Images of {HeLa} Cells},
	journal   = {Bioinformatics},
	volume    = {17},
	number    = {12},
	pages     = {1213--1223},
	year      = {2001},
	doi       = {10.1093/bioinformatics/17.12.1213}
}

@article{Alizadeh:2016,
	title={Measuring systematic changes in invasive cancer cell shape using {Zernike} moments},
	author={Alizadeh, Elaheh and Lyons, Samanthe Merrick and Castle, Jordan Marie and Prasad, Ashok},
	journal={Integrative Biology},
	volume={8},
	pages={1183--1193},
	year={2016},
	publisher={Oxford University Press}
}

@article{Lakshminarayanan:2011,
	author    = {Lakshminarayanan, Vasudevan and Fleck, Andre},
	title     = {{Zernike} Polynomials: {A} Guide},
	journal   = {Journal of Modern Optics},
	volume    = {58},
	number    = {7},
	pages     = {545--561},
	year      = {2011},
	doi       = {10.1080/09500340.2011.554896}
}

@article{Janssen:2007,
	author    = {Janssen, Augustus J. E. M. and Dirksen, Peter},
	title     = {Computing {Zernike} Polynomials of Arbitrary Degree Using the
               Discrete {Fourier} Transform},
	journal   = {Journal of the European Optical Society -- Rapid Publications},
	volume    = {2},
	pages     = {07012},
	year      = {2007},
	doi       = {10.2971/jeos.2007.07012}
}

@article{Shakibaei:2013,
	author    = {Shakibaei Asli, Barmak Honarvar and Paramesran, Raveendran},
	title     = {Recursive Formula to Compute {Zernike} Radial Polynomials},
	journal   = {Optics Letters},
	volume    = {38},
	pages     = {2487--2489},
	year      = {2013},
	doi       = {10.1364/OL.38.002487}
}

@inproceedings{Liao:1998a,
	author    = {Liao, Simon X. and Pawlak, Miroslaw}, 
	title     = {A Study of {Zernike} Moment Computing},
	editor    = {Chin, Roland and Pong, Ting-Chuen},
	booktitle = {Computer Vision --- ACCV'98},
	year      = {1998},
	publisher = {Springer Berlin Heidelberg},
	address   = {Berlin, Heidelberg},
	pages     = {394--401},
}

@article{Kaur:2018,
	author    = {Parminder Kaur and Husanbir Singh Pannu},
	title     = {Comprehensive Review of Continuous and Discrete Orthogonal Moments in Biometrics},
	journal   = {International Journal of Computer Mathematics: Computer Systems Theory},
	volume    = {3},
	pages     = {1--38},
	year      = {2018},
	publisher = {Taylor \& Francis},
	doi       = {10.1080/23799927.2018.1457080}
}

@article{Qi:2021,
	author    = {Shuren Qi and Yushu Zhang and Chao Wang and Jiantao Zhou and Xiaochun Cao},
	title     = {A Survey of Orthogonal Moments for Image Representation: Theory, Implementation, and Evaluation},
	journal   = {{ACM} Computing Surveys},
	volume    = {55},
	pages     = {1--35},
	year      = {2021},
	doi       = {10.1145/3479428}
}

@article{Cui:2025,
	title={Application of {Z}ernike moments for the quantitative analysis of polycyclic aromatic hydrocarbons based on fluorescence three-dimensional spectra},
	author={Cui, Yao-yao and Wu, Shao-zhe and Cui, Can and Wu, Han-bing and Li, Jin-yi and Wang, Tian},
	journal={Analytical Methods},
	volume={17},
	pages={9638--9648},
	year={2025},
}

@article{Singh:2011,
	title={Face recognition using {Z}ernike and complex {Z}ernike moment features},
	author={Singh, Chandan and Mittal, Neerja and Walia, Ekta},
	journal={Pattern Recognition and Image Analysis},
	volume={21},
	pages={71--81},
	year={2011},
}

@article{Soekarta:2025,
	title={Recent Facial Image Preprocessing Techniques: A Review},
	author={Soekarta, Rendra and Ku-Mahamud, Ku Ruhana},
	journal={Engineering Proceedings},
	volume={84},
	pages={39},
	year={2025},
}

@article{Niu:2022,
	title={Zernike polynomials and their applications},
	author={Niu, Kuo and Tian, Chao},
	journal={Journal of Optics},
	volume={24},
	pages={123001},
	year={2022},
	publisher={IOP Publishing}
}

@incollection{Bagherian:2023,
	title={Brain tumor segmentation based on {Z}ernike moments, enhanced ant lion optimization, and convolutional neural network in {MRI} images},
	author={Bagherian Kasgari, Abbas and Ranjbarzadeh, Ramin and Caputo, Annalina and Baseri Saadi, Soroush and Bendechache, Malika},
	booktitle={Metaheuristics and optimization in computer and electrical engineering: Volume 2: Hybrid and improved algorithms},
	pages={345--366},
	year={2023},
	publisher={Springer}
}

@article{Kim:2003,
	title={Invariant image watermark using {Z}ernike moments},
	author={Kim, Hyung Shin and Lee, Heung-Kyu},
	journal={{IEEE} transactions on Circuits and Systems for Video Technology},
	volume={13},
	pages={766--775},
	year={2003},
}

@article{Chen:2023,
	title={A fast method for robust video watermarking based on {Z}ernike moments},
	author={Chen, Shiyi and Malik, Asad and Zhang, Xinpeng and Feng, Guorui and Wu, Hanzhou},
	journal={{IEEE} Transactions on Circuits and Systems for Video Technology},
	volume={33},
	pages={7342--7353},
	year={2023},
}

@article{Zhang:2025,
	title={An improved {SURF} and modified {Z}ernike moments descriptor for object recognition},
	author={Zhang, Lei and Pu, Jiexin and Chen, Gui and Song, Xiaoli},
	journal={Electronics},
	volume={14},
	pages={1025},
	year={2025},
}

@article{Meng:2025,
	title={Protecting the Copyright of Intelligent Transportation Systems Based on {Zernike} Moments},
	author={Meng, Jiale and Lu, Zhe-Ming},
	journal={{IEEE} Transactions on Intelligent Transportation Systems},
	year={2025},
}

@article{Liao:1998b,
	title={On the accuracy of {Zernike} moments for image analysis},
	author={Liao, Simon X and Pawlak, Miroslaw},
	journal={{IEEE} transactions on pattern analysis and machine intelligence},
	volume={20},
	pages={1358--1364},
	year={1998},
}

@article{Huang:2023,
	title={Review of quaternion-based color image processing methods},
	author={Huang, Chaoyan and Li, Juncheng and Gao, Guangwei},
	journal={Mathematics},
	volume={11},
	pages={2056},
	year={2023},
}

@article{Chen:2012,
	title={Quaternion {Zernike} moments and their invariants for color image analysis and object recognition},
	author={Chen, BJ and Shu, HZ and Zhang, Hui and Chen, Gang and Toumoulin, Christine and Dillenseger, Jean-Louis and Luo, Limin M},
	journal={Signal processing},
	volume={92},
	pages={308--318},
	year={2012},
}

@article{Chen:2012b,
	title={Color face recognition using quaternion representation of color image},
	author={Chen, Bei-Jing and Sun, Xing-Ming and Wang, Ding-Cheng and Zhao, Xiao-Ping},
	journal={Acta Autom.\ Sin},
	volume={38},
	pages={1815--1823},
	year={2012}
}

@article{Shapiro:1993,
	title={Embedded image coding using zerotrees of wavelet coefficients},
	author={Shapiro, Jerome M},
	journal={{IEEE} Transactions on signal processing},
	volume={41},
	pages={3445--3462},
	year={1993},
}

@article{Revaud:2008,
	title={Improving {Zernike} moments comparison for optimal similarity and rotation angle retrieval},
	author={Revaud, J{\'e}r{\^o}me and Lavou{\'e}, Guillaume and Baskurt, Atilla},
	journal={{IEEE} transactions on pattern analysis and machine intelligence},
	volume={31},
	number={4},
	pages={627--636},
	year={2008},
}

@software{Pillow,
	author       = {Jeffrey Alex Clark and Contributors},
	title        = {Pillow: A modern fork of {PIL}},
	month        = oct,
	year         = 2015,
	doi          = {10.5281/zenodo.596518},
	url          = {https://github.com/python-pillow/Pillow}
}

@book{Gonzalez:2004,
	title = {Digital image processing with {MATLAB}},
	author = {Gonzalez, Rafael C. and Woods, Richard E. and Eddins, S. L.},
	publisher = {Prentice Hall},
	address = {Upper Saddle River, N.J.},
	year = {2004},
}

@article{Gower:1966,
	author  = {James C. Gower},
	title   = {Some Distance Properties of Latent Root and Vector Methods Used in Multivariate Analysis},
	journal = {Biometrika},
	volume  = {53},
	pages   = {325--338},
	year    = {1966}
}

@book{Borg:2005,
	author    = {Ingwer Borg and Patrick J. F. Groenen},
	title     = {Modern Multidimensional Scaling: Theory and Applications},
	edition   = {2},
	publisher = {Springer},
	address   ={New York, NY},
	year      = {2005}
}

@article{Metz:1978,
	author  = {Metz, Charles E.},
	title   = {Basic Principles of {ROC} Analysis},
	journal = {Seminars in Nuclear Medicine},
	year     = {1978},
	volume   = {8},
	number   = {4},
	pages    = {283--298},
	doi      = {10.1016/S0001-2998(78)80014-2}
}

@article{Hanley:1982,
	author  = {Hanley, James A. and McNeil, Barbara J.},
	title   = {The Meaning and Use of the Area under a Receiver Operating Characteristic ({ROC}) Curve},
	journal = {Radiology},
	year     = {1982},
	volume   = {143},
	number   = {1},
	pages    = {29--36},
	doi      = {10.1148/radiology.143.1.7063747}
}

@article{Montero:2021,
	author={Barrag\'{a}n-Montero, A. and Javaid, U. and Vald\'{e}s, G. and Nguyen, D. and Desbordes, P. and Macq, B. and Willems, S. and Vandewinckele, L. and Holmstr\"{o}m, M. and L\"{o}fman, F. and Michiels, S.}, 
	title={Artificial intelligence and machine learning for medical imaging: A technology review},
	journal={Physica Medica}, 
	volume={83}, 
	pages={242--256},
	year={2021},
}

@article{Castiglioni:2021,
	author={Castiglioni, I. and Rundo, L. and Codari, M. and Di Leo, G. and Salvatore, C. and Interlenghi, M. and Gallivanone, F. and Cozzi, A. and D'Amico, N.C. and Sardanelli, F.}, 
	title={{AI} applications to medical images: From machine learning to deep learning}, 
	journal={Physica medica}, 
	volume={83}, 
	pages={9--24},
	year={2021},
}

@article{Trigka:2025,
	author={Trigka, M. and Dritsas, E.},
	title={A comprehensive survey of deep learning approaches in image processing},
	journal={Sensors}, 
	volume={25}, 
	pages={531},
	year={2025}, 
}

@article{LeCun:2015,
	author    = {LeCun, Yann and Bengio, Yoshua and Hinton, Geoffrey},
	title     = {Deep learning},
	journal   = {Nature},
	volume    = {521},
	number    = {7553},
	pages     = {436--444},
	year      = {2015},
	doi       = {10.1038/nature14539}
}

@incollection{Krizhevsky:2012,
	author    = {Krizhevsky, Alex and Sutskever, Ilya and Hinton, Geoffrey E.},
	title     = {{ImageNet} Classification with Deep Convolutional Neural Networks},
	booktitle = {Advances in Neural Information Processing Systems 25},
	editor    = {Pereira, F. and Burges, C. J. C. and Bottou, L. and Weinberger, K. Q.},
	pages     = {1097--1105},
	publisher = {Curran Associates, Inc.},
	year      = {2012}
}

@article{Jiao:2019,
	author={Jiao, L. and Zhao, J.}, 
	title={A survey on the new generation of deep learning in image processing}, 
	journal={{IEEE} Access}, 
	volume={7}, 
	pages={172231--172263},
	year={2019},
}

\section{Biography Section}

\begin{IEEEbiography}[{\includegraphics[width=1in,height=1.25in,clip,keepaspectratio]{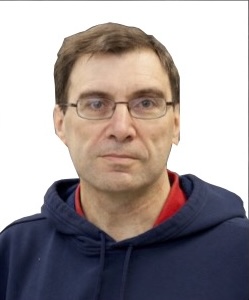}}]{Patrice Koehl}
was trained as an engineer at the Ecole Centrale de Paris,
France, where he graduated with a master degree in Bioengineering in 1984.
After two years working as an intern at the University of California,
Berkeley, he decided to pursue a PhD in Molecular Biology at the University
Louis Pasteur, Strasbourg, France, which he completed in 1989.
He joined the French National Center for Research (CNRS) the same year as a
senior scientist.  In 1997, he visited the department of Structural Biology at Stanford University;
he extended his stay over seven years, becoming a senior research associate in that department.
In 2004, he joined the University of California, Davis, where he is currently Distinguished Professor of Computer Science.
He is a recipient of the Bronze medal of the CNRS and an Alfred P. Sloan fellow.
\end{IEEEbiography}

\begin{IEEEbiography}[{\includegraphics[width=1in,height=1.25in,clip,keepaspectratio]{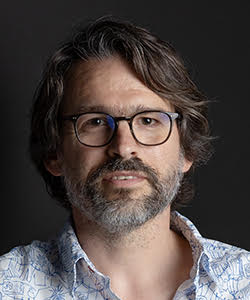}}]{Stephan Tillmann} studied mathematics and philosophy in Mainz (Germany) and Melbourne (Australia). After graduating with a doctoral thesis in low-dimensional topology from the University of Melbourne, he spent three years each in Montr\'eal, Melbourne and Brisbane.  He joined the University of Sydney in 2011 and currently holds the positions of Professor of Geometric Topology and Executive Director of the Sydney Mathematical Research Institute. For his work he has been awarded a Future Fellowship by the Australian Research Council and the Friedrich Wilhelm Bessel Research Award from the Alexander von Humboldt Foundation in 2021.
\end{IEEEbiography}

\end{document}